\documentclass{article}
\usepackage{graphicx} 
\usepackage{amsmath}
\usepackage{amsfonts}
\usepackage{amsthm}
\usepackage{url}
\usepackage{color}
\newtheorem{theorem}{Theorem}[section]
\newtheorem{lemma}[theorem]{Lemma}

\newtheorem{corollary}[theorem]{Corollary}

\newtheorem{remark}[theorem]{Remark}

\usepackage{graphicx}
\usepackage{hyperref}
\usepackage{subfig}
\usepackage{float}
\floatstyle{ruled}
\newfloat{algorithm}{hbt}{alg}[section]
\floatname{algorithm}{Algorithm}
\def \bI {\mathbf{I}}

\def \bv {{\boldsymbol v}}

\def \supp  {\,{\rm supp}\,}

\begin{document}
\title{A Random Integration Algorithm for High-dimensional Function Spaces}
\author{Liang Chen\thanks{Department of Mathematics, Jiujiang University, Jiujiang, 332000,
 P. R. China. E-mail address: {\it chenliang3@alumni.sysu.edu.cn}. Supported in part by National Natural Science Foundation of China under grant 12461019}.
 \quad Minqiang Xu\thanks{Corresponding Author. College of Science, Zhejiang University of Technology, Hangzhou, 310023, P. R. China. E-mail address: {\it mqxu@zjut.edu.cn}. Supported in part by National Natural Science Foundation of China (No. 12326346, 12326347), and by Zhejiang Provincial Natural Science Foundation of
China (No. ZCLY24A0101).}\quad  Haizhang Zhang\thanks{School of Mathematics (Zhuhai), Sun Yat-sen University, Zhuhai, 519082,
 P. R. China. E-mail address: {\it zhhaizh2@sysu.edu.cn}. Supported in part by National Natural Science Foundation of China under grant 12371103, and by Guangdong Basic and Applied Basic Research Foundation (2024A1515011194).}}
\date{}
\maketitle
\begin{abstract}
We introduce a novel random integration algorithm that exhibits a high convergence order for functions characterized by sparse frequencies or rapidly decaying Fourier coefficients. Specifically, for integration in periodic isotropic Sobolev spaces and isotropic Sobolev spaces with compact support, our approach achieves a nearly optimal root mean square error (RMSE) bound.  In contrast to previous nearly optimal algorithms, our method exhibits polynomial tractability. Our integration algorithm
also enjoys nearly optimal bound for weighted Sobolev space.  By incorporating the trick of change of variable, our algorithm is proven to achieve the semi-exponential
convergence order for the integration of analytic functions, which marks a significant
improvement over the previously obtained super-polynomial convergence order. Furthermore, for integration involving Wiener-type functions, the sample complexity of our algorithm remains independent of the decay rate of the Fourier coefficients.

\end{abstract}


\section{Introduction}
This paper is concerned with numerically integrating multivariate functions defined over the $d$-dimensional unit cube. Denote 
$$\operatorname{INT}(f):=\int_{[0,1]^d} f(\boldsymbol{x}) d\boldsymbol{x}.$$
Evaluation of $\operatorname{INT}(f)$ amounts to estimating the value of $\hat{f}(0)$, where the Fourier transform is defined as
\begin{equation}\hat{f}(\boldsymbol{w})=\int_{ [0,1]^d}f(\boldsymbol{x})\exp(-2\pi i \boldsymbol{w} \cdot \boldsymbol{x})d\boldsymbol{x}.
\end{equation} 
For a given positive integer $N$, we introduce the following discrete Fourier transform 
\begin{equation}\label{dft}
\widetilde{f}^{(N)}(\boldsymbol{w})=\frac{1}{N^d}\sum_{ \boldsymbol{z}\in \mathbb{Z}_{N}^{d} }f(\boldsymbol{z}/N)\exp(-2\pi i \boldsymbol{w} \cdot \boldsymbol{z}/N),\quad \boldsymbol{w}\in \mathbb{Z}_{N}^{d},
\end{equation}
where $\mathbb{Z}_{N}^{d}:=\mathbb{Z}^d  \cap [0,N)^d$.

Inspired by the sparse Fourier transform (SFT) from signal processing \cite{chen2022note,gilbert2005improved,hassanieh2012nearly}, our methodology for estimating $\hat{f}(0)$ comprises two fundamental stages. First, we create a hash map that effectively disperses the frequencies, ensuring that the energy of frequencies near zero (excluding zero itself) is small. Second, we employ a low-pass filter to extract the low-frequency components.
Here, we utilize the hash mapping $\boldsymbol{A}_{N}$ developed in our previous paper \cite{chen2022note},
\begin{equation}
\boldsymbol{A}_{N}:=\begin{pmatrix}
\bI_{d-1} & {\bf 0}\\
\bv & h_d
\end{pmatrix},
\end{equation}
where  $\boldsymbol{A}_{N}$ is a $d$ by $d$ matrix, $N$ is some prime number, $\bI_{d-1}$ represents the identity matrix of order $d-1$, $\bv:=(h_1,h_2,\dots,h_{d-1})$, and $(\bv,h_d)$ is drawn from the uniform distribution on the set
$
\mathbb Z^{d}\cap [1,N)^{d}
$. 
Note that the distribution here differs from the setting in \cite{chen2022note}. Subsequently, we construct the following low-pass filter:
\begin{equation}\label{101}
  f_{\boldsymbol{A}_{N},L,r}(\boldsymbol{x}):=\sum_{|l|\le L }   f(\{\boldsymbol{A}_{N}^{\top}(\boldsymbol{x}-\boldsymbol{y}_{l})\}) G_{r,l},  
\end{equation}
where $\{t\}\in [0,1)$ denotes the fractional part of $t$ (for instance, $\{3.2\}=0.2$, $\{-1.3\}=0.7$),   for a vector $\boldsymbol{v}$, $\{\boldsymbol{v}\}$ signifies the operation of taking the fractional part of each component of the vector individually, $\boldsymbol{y}_{l}=(0,\dots,0,l/N)^{\top}$, and
$$G_{r,l}=\frac{1}{r\sqrt{2\pi}}\exp(-\frac{l^2}{2r^2}).$$
The comprehensive attributes of  $G_{r,l}$ are elaborated in Lemma \ref{seeea}. By letting  $\boldsymbol{z}=\boldsymbol{A}_{N}^{\top}\boldsymbol{x}$, and substituting it into (Eq.~(\ref{101})), we derive the following critical integration 
approximation of $\operatorname{INT}(f)$
\begin{equation}\label{1021}
  I(f)_{\boldsymbol{H}_{N},L,r,\boldsymbol{z}}:=\sum_{|l|\le L }   f\left(\left\{\frac{\boldsymbol{z}-l\boldsymbol{H}_{N}}{N}\right\}\right) G_{r,l},
\end{equation}
where $\boldsymbol{H}_{N}$ and $\boldsymbol{z}$ are drawn from the uniform distribution over $\mathbb Z^{d}\cap [1,N )^{d}$ and  $\mathbb Z^{d}\cap [0,N )^{d}$, respectively.

When employing other recent high-dimensional SFT methods for estimating numerical integration, as cited in references \cite{iwen2013improved,kammerer2021high,gross2022sparse}, it is not feasible to achieve the desired error bounds for functions in Sobolev spaces. Previously, Gilbert et.al \cite{gilbert2005improved}  introduced an algorithmic framework for high-dimensional SFT, which could potentially be adapted to develop integration algorithms that attain the desired error bound. However, such a framework would increase the upper bound by a factor of  $d^5$, as analyzed on page 61 of \cite{kammerer2021high}.   We remark that (Eq.~(\ref{1021})) exhibits a certain similarity to the quadrature formula introduced in \cite{haselgrove1961method}.

Utilizing (Eq.~(\ref{1021})), along with the local Monte Carlo sampling (see (Eq.~(\ref{MCC}))) and the median trick \cite{jerrum1986random,niemiro2009fixed,kunsch2019optimal,goda2022construction,pan2023super,pan2024super,goda2024universal,pan2025l_2,pan2025dimension,ye2025median}, we shall construct a novel randomized integration algorithm. Our method introduces a fresh understanding in the construction of numerical integration algorithms from the perspectives of filtering and dispersing sparse frequencies. The algorithm possesses several theoretical advantages compared to existing algorithms. We summarize these advantages as follows:
\begin{itemize}
\item {\it Our algorithm is nearly optimal for integration in the periodic isotropic Sobolev space as well as in the isotropic Sobolev space with compact support, 
while preserving polynomial tractability.
 Consequently, it is also nearly optimal and
polynomially tractable for smooth functions that vanish on the boundary. In contrast, previous works
\cite{krieg2017universal,ullrich2017Monte} have developed optimal integration algorithms for the mixed Sobolev space and
the isotropic Sobolev space with compact support. However, these algorithms suffer from a drawback: their RMSE upper bounds are not less than $2^{\Theta(d)}M^{s/d+1/2}$ (where $d$, $s$, and $M$ denotes the dimension, the order, and the sample size,  respectively, see Remark 3.2 in \cite{ullrich2017Monte}. 
In \cite{krieg2017universal}, the corresponding constant is larger than $1/(d^{d/2}(2^{1/d}-1)^d)$, and the order of the Sobolev space is required to be greater than $d/2$,  as observed from the proofs of Theorem 1 and Lemma 5 in \cite{krieg2017universal}), rendering them not polynomially tractable in terms of sample complexity. Furthermore, our upper bound for integration in the isotropic Sobolev space with compact support is superior to that presented in \cite{krieg2017universal,ullrich2017Monte} when $s\le \Theta(d)$ and $M\le 2^{2^{\Theta(d)}}$ (namely, $\log^{\Theta(1)} M \le 2^{\Theta(d)}$). Recently, \cite{chopin2024higher} has proposed a nearly optimal integration algorithm for general smooth functions and smooth functions that vanish on the boundary of $[0,1]^d$ (see Theorem 3.2 in \cite{chopin2024higher}). However, as acknowledged  in the ``Future work" section of  \cite{chopin2024higher}, they still struggle to overcome the curse of dimensionality in terms of sample complexity. For more details, see  Theorem \ref{T2}, Remark \ref{28888}, Corollary \ref{corollv}, and Remark \ref{2112233} in Section 3.}

\item {\it By incorporating the trick of change of variables,  the numerical integration of the analytic functions over $[0,1]^d$ is proven to achieve semi-exponential convergence order, which marks a significant improvement over the previously obtained super-polynomial convergence order \cite{pan2023super,pan2024super,suzuki2017super,dick2017construction} (our algorithmic analysis also applies to the infinitely differentiable functions defined in \cite{suzuki2017super,dick2017construction}). For more details, see Theorem \ref{ta1} and Remark \ref{jiexi} in Section 4.}

\item {\it Our algorithm achieves a convergence rate of order $\frac{(\log(M))^{ 2 }}{M}$ (where $M$ denotes the sample size) for integration in the subspace of Wiener algebra. In comparison to previous studies\cite{Dickwl,dick2017digital,goda,goda2023strong,chen2024information,krieg2024tractability,dick2024some}, the sample complexity of our algorithm is independent of the decay rate of Fourier coefficients and the Hölder continuity of the functions, provided that we choose a sufficiently large $N$ in (Eq.~(\ref{1021})). For more details, see Theorem \ref{weina} and Remarks \ref{recon}, \ref{reweina} in  Section 4.}

\item {\it Our algorithm is also nearly optimal for integration in weighted Sobolev space.      Unlike the algorithms \cite{kritzer2019lattice,dick2022component,kuo2023random,nuyens2023randomised,goda2024randomized}, our algorithm does not require prior information about weights, making it universally applicable to varying weights in the RMSE sense. The integration algorithms applicable to different weights under the worst-case scenario have been established in \cite{dick2006good,dick2012random,dick2021stability,dick2022lattice,goda2022construction}. For more details see Theorem \ref{T3} and Remark \ref{piu} in Section 5. Very recently,
\cite{goda2024simple} has presented a universal (with respect to both weights and order) nearly optimal algorithm for weighted Sobolev spaces in the scenario where the order $s>1/2$, whereas our nearly optimal algorithm is valid for the case where $s>0$. These two works are independent of each other and  different in terms of algorithms, the earliest version (for the ``weighted Korobov space'' of order $s>0$) of our paper has already been uploaded to arXiv \cite{chen2024random}. For more details, see Remark \ref{goda2}.
} 
\end{itemize}

 At the end of the introduction, we collect some notations that are frequently used in the paper. 
\begin{itemize}
    \item $\mathbb{N}$:\ the set of natural numbers; $\mathbb{N}_{0}=\mathbb{N} \cup \{0\}$; $\mathbb{Z}$: the set of all integers; For $N\in \mathbb{N}$, $\mathbb{Z}_{N}=\mathbb{Z}\cap [0,N)$.

    \item The torus is $\mathcal{T}=\mathbb{R}/\mathbb{Z} = [0,1)$. The symbol $\log$ refers to the natural logarithm. 

    \item $\#U$: the cardinality of a set $U$; $\mathbb{I}_{U}$: the indicator function of $U$.

    \item For two nonnegative functions $h$ and $g$ defined on a common domain, we express $h(\tau) = \Theta(g(\tau))$ if there exist positive  absolute  constants $C_1$ and $C_2$ such that $C_{1}g(\tau) \le h(\tau) \le C_{2}g(\tau)$ holds for all $\tau$ within the domain.

    \item $w \in [-b,b]\, ({\rm mod }\ N)$: if there exists an integer $a$ such that $w + aN \in [-b,b]$; $w \notin [-b,b]\, ({\rm mod }\ N)$ if no such integer $a$ exists.

    \item $\boldsymbol{\xi}\sim \mbox{unif}\, \Omega$: $\boldsymbol{\xi}$ is a random variable uniformly distributed over the set $\Omega$.

    \item For any function $f: \mathbb{Z}_{N}^d \rightarrow \mathbb{C}$, we define its expectation as $$\mathbb{E}_{\boldsymbol{z}} f(\boldsymbol{z}) = \frac{1}{N^d} \sum_{\boldsymbol{z} \in \mathbb{Z}_{N}^d} f(\boldsymbol{z}).$$ 

    \item Given $\boldsymbol{H}_{N} \in \mathbb{Z}^{d} \cap [1,N)^{d}$ and $B>0$, $U_{\boldsymbol{H}_{N},B} := \{\boldsymbol{w} \mid \boldsymbol{H}_{N} \cdot \boldsymbol{w} \in [-N/B, N/B]\, ({\rm mod}\  N)\}$.

    \item For a vector $\boldsymbol{k} = (k_1, k_2, \cdots, k_d) \in \mathbb{Z}^{d} \setminus \{\boldsymbol{0}\}$, we define the support of $\boldsymbol{k}$ as $\supp(\boldsymbol{k}) = \{j|  j\in\{1,2,\dots,d\},k_{j} \neq 0\}$. 

    \item The median of a set of complex numbers $\{a_{j}\}_{j=1}^k$ is defined as $${\rm median} \{a_{j}\}_{j=1}^k = {\rm median}\{\Re a_{j}\}_{j=1}^k + i \cdot {\rm median}\{\Im a_{j}\}_{j=1}^k,$$ where ${\rm median}\{\Re a_{j}\}_{j=1}^k$ and ${\rm median}\{\Im a_{j}\}_{j=1}^k$ denote the medians of the real and imaginary parts of the set, respectively.

    \item In this paper, all hidden constants are considered to be absolute constants.

     \item Our code is available at
 https://github.com/Liang-CL-12/sparse-integration.
\end{itemize}

\section{Integration  Error  from  Sparsity of Frequencies} 
In this section, we characterize the integration error of our algorithm in terms of sparsity. The integration error to be established in Theorem \ref{T1} demonstrates that when the sample size  is of order {$n(\log(1/\epsilon))(\log(n/\epsilon))$, we can filter out the $n$ nonzero frequency components with min-bandwidth (as defined in (Eq.~\eqref{BDW})) bounded by $\mathcal{M}$ with the error $\epsilon$. This suggests that the min-bandwidth of the sparse frequency components does not affect the integration error, provided that the sample size is slightly larger than the sparsity. For the remaining part, we achieve an error on the order of $(\log(1/\epsilon))^{2}/\sqrt{M}$ (where $M$ denotes the sample size), which nearly matches the error rate of standard Monte Carlo methods. 

  We first present some preliminary notions and results. For a function $f$, we consider the random integration algorithms with the random variable
$\omega$ of the following form
\begin{equation*}
A_{d,n,w}(f)=\phi_{d,n,w}(f(\boldsymbol{x}_1),f(\boldsymbol{x}_2),\cdots,f(\boldsymbol{x}_n)),
\end{equation*}
where $\{\boldsymbol{x}_{j}\}_{j=1}^{n}\subseteq [0,1]^d$ is a sample set and $\phi_{d,n,w}$ is a map $\mathbb{C}^{n}\rightarrow \mathbb{C}$.

For an accuracy $\epsilon \in (0,1),$ the information complexity of numerical integration in a function space $F$ in the random setting is defined by
$$
  N(\epsilon,F,d):=\inf\Big\{ n\in \mathbb{N} \big| \exists A_{d,n,w}: E(F,A_{d,n,w})\leq \epsilon \Big\},
 $$
where \[ E(F,A_{d,n,w}):=\sup_{\substack{f\in F\\ \|f\|_{F}\leq 1}}\bigg(\mathbb{E}_{\omega}\Big| I_d(f)-A_{d,n,w}(f)\Big|^{2}\bigg)^{1/2}.\]
We say the random algorithms is polynomially tractable, if $N(\epsilon,F,d)$  is polynomial with respect to $d$ and $1/\epsilon$. Furthermore, if $N(\epsilon,F,d)$  is independent of $d$, we say the algorithm is strong polynomially tractable.
For further information on  tractability in various settings, see \cite{Novak1,Novak2,Novak3} by Novak and Woźniakowski.

Next, we need some technical lemmas.  According to Proposition 2.2 in \cite{kunsch2019solvable}, we have the following corollary.

\begin{lemma}\label{Median}
 Let $\mathcal{W}$ be a given function space with semi-norm $\|\cdot\|_{\mathcal{W}}$. If $R_m$ is a random algorithm such
that
 $$
\sup_{\|f\|_{\mathcal{W}} \leq 1} \mathbb{P}\bigg\{\Big|R_m(f)-\operatorname{INT} (f)\Big|>\epsilon\bigg\} \leq \alpha,
$$
 where $0<\epsilon,\alpha<1/2$, then for an odd positive integer $k$, it holds
 $$
\sup _{\|f\|_{\mathcal{W}} \leq 1} \mathbb{P}\Bigg\{\bigg|{\rm median}\Big\{R_{m, j}(f)\Big\}_{j=1}^{k}-\operatorname{INT} (f)\bigg|>2\epsilon\Bigg\} \le  (4\alpha(1-\alpha))^{k/2},$$
where $\{R_{m, j}(f)\}_{j=1}^k$ is a set of $k$ independent realizations of $R_m(f)$.
\end{lemma}
Based on Claim 7.2 in \cite{hassanieh2012nearly} (see the full version: arXiv:1201.2501v2), we have the following lemma. In order to ensure the paper's self-sufficiency, we present a proof. 
\begin{lemma}\label{biu} Let $G_{r,l,\infty}:=\sum_{ k\in \mathbb{Z}}\frac{1}{r\sqrt{2\pi}}\exp\left(-\frac{(l+kN)^2}{2r^2}\right).$
For any $w\in \mathbb{Z}_{N}$, it holds 
\begin{equation}\label{l22e10}
\sum_{l\in \mathbb{Z}_N}G_{r,l,\infty} \exp(2 \pi i wl/N)  =\sum_{k\in  \mathbb{Z}} \exp\left(-2 (\pi r (k+w/N))^2\right).
\end{equation}
\end{lemma}
\begin{proof}
 Recall the Fourier transform over $\mathbb{R}$
\begin{equation*}\widehat{f}_{\mathbb{R}}(w)=\int_{-\infty}^{+\infty}f(x)\exp(-2\pi i w  x)dx.
\end{equation*} 
 Let $g_{r}(x)=\frac{1}{r\sqrt{2\pi}}\exp(-\frac{x^2}{2r^2})$.
It is well-known that $\widehat{g_{r}}_{\mathbb{R}}(w)= \exp(-2(\pi r w)^{2}) $ . Let $G_{r,\infty}(x)=\sum_{ k\in \mathbb{Z}}\frac{1}{r\sqrt{2\pi}}\exp(-\frac{(x+kN)^2}{2r^2})$.
 By Poisson's summation formula (see Theorem 2.28 in \cite{plonka2018numerical}),  
 \begin{equation}\label{poss}
G_{r,\infty}(x)=\frac{1}{N}\sum_{w\in \mathbb{Z}}\exp(-2(\pi r w/N)^{2})\exp(2 \pi i w x/N).
\end{equation} 
Note that $$\sum_{l\in \mathbb{Z}_{N}}\exp(2 \pi i w'l/N)\exp(-2 \pi i wl/N)=N \quad for \quad w'= w ({\rm mod} N),\quad w',w\in \mathbb{Z},$$ and $$\sum_{l\in \mathbb{Z}_{N}}\exp(2 \pi i  w' l/N)\exp(-2 \pi i wl/N)=0\quad for \quad w'\neq w ({\rm mod} N),\quad w',w\in \mathbb{Z}.$$
 Multiplying both sides of (Eq.~(\ref{poss})) by $\exp(2\pi i wx/N)$ and  substituting $x=l$, we have 
\begin{equation*}\begin{split}
\sum_{l\in \mathbb{Z}_N}G_{r,l,\infty} \exp(2 \pi i wl/N)  &=\sum_{k\in  \mathbb{Z}} \exp\left(-2 \left(\pi r (k+w/N)\right)^2\right).
\end{split}
\end{equation*}
\end{proof}
\begin{lemma}\label{seeea}
Let $0<\epsilon\le1/10$,  $B$ be an integer greater than 2 , $r=B \sqrt{ \log (1 / \epsilon)}$, and $L=\lceil r\sqrt{2\log (1/\epsilon)}\rceil$. If $N$ is an integer that satisfies $N\ge 3L$ then 
\begin{equation}\label{l23e10}
\left|\sum_{|l|\le L}G_{r,l}\exp(2 \pi i wl/N)\right|\le 10\epsilon \quad~\text{for~any ~}w\notin [-N/B,N/B]\, ({\rm mod}N),
\end{equation}
and 
\begin{equation}\label{l23e11}
\left|1-\sum_{|l|\le L}G_{r,l}\right|\le 10\epsilon .
\end{equation}
 \end{lemma}

\begin{proof}
Given that $ N\ge 3L$ and  $L=\lceil r\sqrt{2\log (1 / \epsilon)}\rceil$, we have for $l\in[-L,L]$, 

 \begin{equation*}
\begin{split}
&|G_{r,l}-G_{r,l,\infty}|\le 2\sum_{k=1}^{+\infty}\frac{1}{r\sqrt{2\pi}}\exp\left(-\frac{\Big(N-|l|+(k-1)N\Big)^2}{2r^2}\right)\\&\le 2\sum_{k=1}^{+\infty}\frac{1}{r\sqrt{2\pi}}\exp\left(-\frac{\Big(N-L+(k-1)N\Big)^2}{2r^2}\right)\le 2\sum_{k=1}^{+\infty}\frac{1}{r\sqrt{2\pi}}\exp\left(-\frac{\Big(2L+(k-1)N\Big)^2}{2r^2}\right)\\
&\le \frac{2}{r\sqrt{2\pi}}\exp\left(-\frac{4L^2}{2r^2}\right) \sum_{k=1}^{+\infty}\exp\left(-\frac{4L(k-1)N}{2r^2}\right) \le \frac{2}{r\sqrt{2\pi}}\exp\left(-\frac{4L^2}{2r^2}\right) \sum_{k=1}^{+\infty}\exp\left(-4(k-1)\right) \\&= \frac{2}{r\sqrt{2\pi}}\exp\left(-\frac{L^2}{2r^2}\right)\exp\left(-\frac{3L^2}{2r^2}\right)\frac{1}{1-\exp(-4)}.  
\end{split}
\end{equation*} 
Since $L/r \ge \sqrt{2\log(1/\epsilon)}$, we have $\exp(-\frac{3L^2}{2r^2})\le \epsilon^2$. It follows that for $l\in[-L,L]$,
\begin{equation*}
\begin{split}
|G_{r,l}-G_{r,l,\infty}|&\le \frac{2}{r\sqrt{2\pi}}\exp(-\frac{L^2}{2r^2})\exp(-\frac{3L^2}{2r^2})\frac{1}{1-\exp(-4)} \\&\le \frac{2}{r\sqrt{2\pi}}\exp(-\frac{L^2}{2r^2}) \frac{\epsilon^{2}}{1-\exp(-4)} \\&\le \frac{\epsilon}{r\sqrt{2\pi}}\exp(-\frac{L^2}{2r^2}).
%
\end{split}
\end{equation*}
Denote the set $\mathbb{Z}_N \backslash([0,L]\cup [N-L,N-1])$ by $M$. Note that

 \begin{equation*}
\begin{split}
\sum_{l\in M}G_{r,l,\infty}&\le 2\sum_{L<l\le \frac{N}{2}}G_{r,l,\infty} \le 4\sum_{L<l\le \frac{N}{2}}\sum_{ k=0}^{+\infty}\frac{1}{r\sqrt{2\pi}}\exp\left(-\frac{(l+kN)^2}{2r^2}\right)\\&\le 4\sum_{L<l\le \frac{N}{2}}\sum_{ k=0}^{+\infty}\frac{1}{r\sqrt{2\pi}}\exp\left(-\frac{l^2+2lkN}{2r^2}\right)\\&\le 4\sum_{L<l\le \frac{N}{2}}\exp\left(-\frac{l^2}{2r^2}\right)\sum_{ k=0}^{+\infty}\frac{1}{r\sqrt{2\pi}}\exp(-\frac{2lkN}{2r^2})\\&\le 4\sum_{L<l\le \frac{N}{2}}\exp\left(-\frac{l^2}{2r^2}\right) \frac{1}{r\sqrt{2\pi}}\frac{1}{1-\exp(-lN/r^2)}\\&\le \frac{4}{r\sqrt{2\pi}}\frac{1}{1-\exp(-LN/r^2)}\sum_{j=1}^{+\infty}\exp\left(-\frac{(L+j)^2}{2r^2}\right) \\&\le \frac{4}{r\sqrt{2\pi}}\frac{1}{1-\exp(-LN/r^2)}\sum_{j=1}^{+\infty}\exp(-\frac{L^2+2jL}{2r^2}) \\&= \frac{4}{r\sqrt{2\pi}}\frac{\exp(-L^2/2r^2)}{1-\exp(-LN/r^2)}\sum_{j=1}^{+\infty}\exp(-\frac{jL}{r^2}) \\&= \frac{4}{r\sqrt{2\pi}}\frac{\exp(-L^2/2r^2)}{1-\exp(-LN/r^2)}\frac{1}{1-\exp(-L/r^2)}. 
\end{split}
\end{equation*} 
Since $L/r \ge \sqrt{2\log(1/\epsilon)}$ and $N>3L$, we have \begin{equation}\label{buchong}
\exp\left(-\frac{L^2}{2r^2}\right)\le \epsilon,\quad \exp\left(-\frac{NL}{r^2}\right)\le \epsilon^2.
\end{equation} Since $L=\lceil r\sqrt{2\log (1 / \epsilon)}\rceil$, $r\ge 2\sqrt{ \log (1 / \epsilon)}$,  and $1-\exp(-x)\ge |x|\exp(-x)$ for $x>0$, we have $$L/r^{2}\le 1\quad \text{and}\quad\frac{1}{1-\exp(-L/r^2)} \le \exp(L/r^2){r^2}/L\le  er^2/L.$$ It follows that
\begin{equation*}
\begin{split}
\sum_{l\in M}G_{r,l,\infty}&\le \frac{4}{r\sqrt{2\pi}}\frac{\exp(-L^2/2r^2)}{1-\exp(-LN/r^2)}\frac{1}{1-\exp(-L/r^2)}\\& \le \frac{4er}{L\sqrt{2\pi}}\frac{\epsilon}{1-\epsilon^2}\le 8\epsilon.
\end{split}
\end{equation*}
Therefore, by employing (Eq.~(\ref{buchong})),  $0<\epsilon \le 1/10$, and  $L=\lceil r\sqrt{2\log (1 / \epsilon)}\rceil$, we derive  that
\begin{equation}\begin{split}\label{l22e1}
&\Big|\sum_{|l|\le L}G_{r,l}\exp(2 \pi i wl/N)-\sum_{l\in \mathbb{Z}_N}G_{r,l,\infty} \exp(2 \pi i wl/N)\Big|\\=&\Big|\sum_{|l|\le L}G_{r,l}\exp(2 \pi i wl/N)-\sum_{|l|\le L}G_{r,l,\infty} \exp(2 \pi i wl/N)-\sum_{l\in M}G_{r,l,\infty}\exp(2 \pi i wl/N)\Big|\\
 \le &  \frac{\epsilon (2L+1)}{r\sqrt{2\pi}}\exp(-\frac{L^2}{2r^2})+8\epsilon \le \frac{5\epsilon^2 \sqrt{\log(1/\epsilon)}}{\sqrt{2\pi}}+8\epsilon\le 9\epsilon .
\end{split}
\end{equation}
Using Lemma \ref{biu} and taking $w=0$ in the (Eq.~\eqref{l22e10}), we have 
\begin{equation}\label{l22e3}\left|\sum_{l\in \mathbb{Z}_N}G_{r,l,\infty} -1\right|=\left|\sum_{k\in  \mathbb{Z}} \exp\left(-2 (\pi r k)^2\right)-1\right|= \left|2\sum_{k=1}^{\infty} \exp\left(-2 (\pi r k)^2\right)\right|\le \epsilon .\end{equation}
Combining (Eq.~\eqref{l22e1}), (Eq.~\eqref{l22e3}), and the triangle inequality yields
\begin{equation*}
\Big|1-\sum_{|l|\le L}G_{r,l}\Big|=\Big|1-\sum_{l\in \mathbb{Z}_N}G_{r,l,\infty}+\sum_{l\in \mathbb{Z}_N}G_{r,l,\infty}-\sum_{|l|\le L}G_{r,l}\Big|\le 10\epsilon .
\end{equation*}
On the other hand, using the (Eq.~\eqref{l22e10}) and considering the condition $r=B \sqrt{ \log (1 / \epsilon)}$, we derive that for  $w\notin [-N/B,N/B] ({\rm mod}\ N)$, 
\begin{equation*}
\begin{split}
&~~~~\left|\sum_{l\in \mathbb{Z}_N}G_{r,l,\infty} \exp(2 \pi i wl/N) \right| =\left|\sum_{k\in  \mathbb{Z}} \exp\left(-2 (\pi r (k+w/N))^2\right)\right|\\
&\le 2\sum_{k=0}^{+\infty} \exp\left(-2 (\pi r (k+1/B))^2\right)\le \epsilon.
\end{split}
\end{equation*}
Combining the above inequality and (Eq.~\eqref{l22e1}), we complete the proof (Eq.~\eqref{l23e10}).
\end{proof} 
The lemma presented below is closely related to Lemma 3.2 in \cite{chen2022note} and Lemma 4 in \cite{kritzer2019lattice}.
\begin{lemma}\label{seeeaa}
Let $N$ be a prime number and $B$ be a positive integer with $N\ge 3B>6$. Then, for any  $\boldsymbol{w}\in \mathbb{Z}^{d}$ with $\boldsymbol{w}\neq 0\, ({\rm mod}\  N)$, it holds
$$\boldsymbol{H}_{N}\cdot \boldsymbol{w}\in[-N/B, N/B]\,({\rm mod} N),$$
with probability at most $4/B$ over the randomness of $\boldsymbol{H}_{N}$, where $\boldsymbol{H}_{N}:=(h_1,h_2,\dots,h_d)$ is drawn from the uniform distribution over
$\mathbb Z^{d}\cap [1,N )^{d}.
$
\end{lemma}
\begin{proof}
Since $\boldsymbol{w}\neq \boldsymbol{0}\,({\rm mod} \ N)$, it must have a non-zero component.
Without loss of generality, let us assume that $w_d\neq 0\, ({\rm mod} \ N)$. Then, fixing arbitrary values $
h_1,\dots, h_{d-1}$, with probability at most $
4/B$ (no greater than  
$\frac{\#\{w\in\mathbb{Z}|w\in[-N/B,N/B]\}}{N-1}$ ) over the randomness of $h_d$,
we have 
$$
\sum_{i=1}^{d}h_{i}w_{i}\in[-N/B, N/B]\,({\rm mod}\  N),
$$
thereby completing the proof.
\end{proof}
Now, we introduce the local Monte Carlo sampling.
For any function $f:[0,1]^d\rightarrow \mathbb{C}$, we introduce a new function $f_{N,\boldsymbol{\eta}}$ defined as
\begin{equation}\label{MCC}
f_{N,\boldsymbol{
\boldsymbol{\eta}}}(\boldsymbol{z}/N):=f(\boldsymbol{z}/N+\boldsymbol{\boldsymbol{\eta}}_{\boldsymbol{z}}),
\end{equation}
where $\boldsymbol{z}\in \mathbb{Z}_{N}^{d}$, $\boldsymbol{\eta}=(\boldsymbol{\eta}_{\boldsymbol{z}})_{\boldsymbol{z}\in\mathbb{Z}^d}$, and each $\boldsymbol{\eta}_{\boldsymbol{z}}$ is independently and identically distributed from the uniform distribution on $[0,1/N)^d$.

 Given a function $f$ with $\|f\|_{L^2([0,1]^d)}\le 1$, we use $a_{\boldsymbol{\xi}}$ to denote the Fourier coefficient of $f$. Then $f$ can be expressed as 
\begin{equation*}
f(\boldsymbol{x})=\sum_{\boldsymbol{\xi}\in \mathbb{Z}^d}a_{\boldsymbol{\xi}}\exp(2\pi i \boldsymbol{\boldsymbol{\xi}}\cdot \boldsymbol{x}).
\end{equation*}
For $\mathcal{M}>0$, let \begin{equation}\label{BDW}
 U\subseteq \{\boldsymbol{w}=(w_1,w_2,\dots,w_d)\in \mathbb{Z}^d \backslash\{\boldsymbol{0}\} | \min_{j\in\supp(\boldsymbol{w})}|w_{j}|\le\mathcal{M}\}  ,\end{equation} and define $g(\boldsymbol{x})$ and $R(\boldsymbol{x})$ as 
\begin{equation}\label{shuoming2}
g(\boldsymbol{x})=\sum_{\boldsymbol{\xi}\in U}a_{\boldsymbol{\xi}}\exp(2\pi i \boldsymbol{\boldsymbol{\xi}}\cdot \boldsymbol{x}),~R(\boldsymbol{x})=\sum_{\boldsymbol{\xi}\notin (U\cup\{\boldsymbol{0}\})}a_{\boldsymbol{\xi}}\exp(2\pi i \boldsymbol{\boldsymbol{\xi}}\cdot \boldsymbol{x}).
\end{equation}
We call $\mathcal{M}$ the min-bandwidth of $g$. 

Then the function $f$ can be rewritten as
\begin{equation}\label{NDf}
f(\boldsymbol{x})=a_{\boldsymbol{0}}+g(\boldsymbol{x})+R(\boldsymbol{x}),\quad \boldsymbol{x}\in [0,1]^d.
\end{equation}
By the definition of $f_{N,\boldsymbol{\eta}}$ in (Eq.~\eqref{MCC}),  we obtain 
$$
f_{N,\boldsymbol{\eta}}(\boldsymbol{z}/N)=a_{\boldsymbol{0}}+g(\boldsymbol{z}/N+\boldsymbol{\boldsymbol{\eta}}_{\boldsymbol{z}})+R(\boldsymbol{z}/N+\boldsymbol{\boldsymbol{\eta}}_{\boldsymbol{z}}).
$$
For convenience, we let  $a_{\boldsymbol{0}}^{*}=\widetilde{f}_{N,\boldsymbol{\eta}}^{(N)}(\boldsymbol{0})=a_{\boldsymbol{0}}+\widetilde{g}_{N,\boldsymbol{\eta}}^{(N)}(\boldsymbol{0})+\widetilde{R}_{N,\boldsymbol{\eta}}^{(N)}(\boldsymbol{0})$, and 
\begin{equation}\label{Rnyt}
R_{2}(\boldsymbol{z}/N)
=R_{N,\boldsymbol{\eta}}(\boldsymbol{z}/N)+g_{N,\boldsymbol{\eta}}(\boldsymbol{z}/N)-g(\boldsymbol{z}/N)-\widetilde{g}_{N,\boldsymbol{\eta}}^{(N)}(\boldsymbol{0})-\widetilde{R}_{N,\boldsymbol{\eta}}^{(N)}(\boldsymbol{0}).
\end{equation}
 Then
\begin{equation}\label{sshuoming}
\widetilde{R_2}^{(N)}(\boldsymbol{0})=\widetilde{R}_{N,\boldsymbol{\eta}}^{(N)}(\boldsymbol{0})+\widetilde{g}_{N,\boldsymbol{\eta}}^{(N)}(\boldsymbol{0})-0-\widetilde{g}_{N,\boldsymbol{\eta}}^{(N)}(\boldsymbol{0})-\widetilde{R}_{N,\boldsymbol{\eta}}^{(N)}(\boldsymbol{0})=0,
\end{equation} 
and $f_{N,\boldsymbol{\eta}}$ can be rewritten as
\begin{equation}\label{fnyt}
f_{N,\boldsymbol{\eta}}(\boldsymbol{z}/N)=a_{\boldsymbol{0}}^*+g(\boldsymbol{z}/N)+R_{2}(\boldsymbol{z}/N).
\end{equation}

\begin{lemma}\label{3lem} Let $f$ and $f_{N,\boldsymbol{\eta}}$ be functions defined in (Eq.~\eqref{NDf}) and (Eq.~\eqref{fnyt}). Let $K=\#U$ and $\|f\|_{L^{2}([0,1]^d)}\le 1$, we have over the randomness of $\boldsymbol{\eta}$ that it holds 
\begin{equation*} 
\mathbb{P}\Big\{|a_{\boldsymbol{0}}^{*}-a_{\boldsymbol{0}}|\le \epsilon\Big\}>1-\frac{1}{\epsilon^{2} N^d},
\end{equation*}
and 
\begin{equation*} 
\mathbb{P}\left\{\frac{1}{N^d}\sum_{\boldsymbol{z}\in \mathbb{Z}_{N}^d}\left|R_{2}(\boldsymbol{z}/N)\right|^2\le \frac{4}{\sigma_2} \|R\|_{L^{2}([0,1]^d)}^{2}+ \Theta\left(\epsilon^2+ \frac{K  }{\sigma_{2} N^{d/2}} \right)\right\}\ge 1-\sigma_2-\frac{1}{\epsilon^{2} N^d}.
\end{equation*}
\end{lemma}
\begin{proof}
Since $\mathbb{E}_{\boldsymbol{\eta}}\widetilde{f}_{N,\boldsymbol{\eta}}^{(N)}(\boldsymbol{0})=a_{\boldsymbol{0}}$ and $\|f\|_{L^{2}([0,1]^d)}\le 1$, we have
\begin{equation}\label{lem25101}
\mathbb{E}_{\boldsymbol{\eta}}\left|\widetilde{f}_{N,\boldsymbol{\eta}}^{(N)}(\boldsymbol{0})-a_{\boldsymbol{0}}\right|^2=\mathbb{E}_{\boldsymbol{\eta}}\left|\frac{1}{N^d}\sum_{\boldsymbol{t}\in \{\boldsymbol{z}/N|z\in \mathbb{Z}_{N}^{d} \}}f_{N,\boldsymbol{\eta}}(\boldsymbol{t})-\frac{1}{N^d}\sum_{\boldsymbol{t}\in \{\boldsymbol{z}/N|z\in \mathbb{Z}_{N}^{d} \}}\mathbb{E}_{\boldsymbol{\eta}}f_{N,\boldsymbol{\eta}}(\boldsymbol{t})\right|^{2}\le \frac{1}{N^d}.\end{equation}
Then by Chebyshev's inequality,
\begin{equation}\label{lem2500}
\mathbb{P}\left\{|a_{\boldsymbol{0}}^{*}-a_{\boldsymbol{0}}|\le \epsilon\right\}=\mathbb{P}\bigg\{\Big|\widetilde{f}_{N,\boldsymbol{\eta}}^{(N)}(\boldsymbol{0})-a_{\boldsymbol{0}}\Big|\le \epsilon\bigg\}=\mathbb{P}\left\{ \left|\widetilde{g}_{N,\boldsymbol{\eta}}^{(N)}(\boldsymbol{0})+\widetilde{R}_{N,\boldsymbol{\eta}}^{(N)}(\boldsymbol{0})\right|\le \epsilon\right\}>1-\frac{1}{\epsilon^{2} N^d}.
\end{equation}
Notice that 
\begin{equation}\label{lem2501add}
R_{2}(\boldsymbol{z}/N)
=R_{N,\boldsymbol{\eta}}(\boldsymbol{z}/N)+\left(g_{N,\boldsymbol{\eta}}(\boldsymbol{z}/N)-g(\boldsymbol{z}/N)\right)-\left(\widetilde{g}_{N,\boldsymbol{\eta}}^{(N)}(\boldsymbol{0})+\widetilde{R}_{N,\boldsymbol{\eta}}^{(N)}(\boldsymbol{0})\right).
\end{equation}
Next, we will estimate the first two terms in (Eq.~\eqref{lem2501add}) respectively. Note that
\begin{equation*}
\begin{split}
\mathbb{E}_{\boldsymbol{\eta}}\frac{1}{N^d}\sum _{\boldsymbol{z}\in \mathbb{Z}_{N}^d}\left|R_{N,\boldsymbol{\eta}}(\boldsymbol{z}/N)\right|^2 =\frac{1}{N^d}\sum _{\boldsymbol{z}\in \mathbb{Z}_{N}^d}\mathbb{E}_{\boldsymbol{\eta_{\boldsymbol{z}}}}\left|R_{N,\boldsymbol{\eta}}(\boldsymbol{z}/N)\right|^2 =\sum _{\boldsymbol{z}\in \mathbb{Z}_{N}^d}
\|R\|_{L^{2}\left(\frac{\boldsymbol{z}}{N}+[0,\frac{1}{N}]^{d}\right)}^2=\|R\|_{L^{2}([0,1]^{d})}.
\end{split}
\end{equation*}
Then for a positive number $\sigma_{2}$, the application of Markov’s inequality yields
\begin{equation}\label{lem2501}
\mathbb{P}\left\{\frac{1}{N^d}\sum _{\boldsymbol{z}\in \mathbb{Z}_{N}^d}\left|R_{N,\boldsymbol{\eta}}(\boldsymbol{z}/N)\right|^2\le \frac{2}{\sigma_{2}}\|R\|_{L^{2}([0,1]^d)}^{2}\right\}>1-\sigma_{2}/2.
\end{equation}
Analogous to (Eq.~\eqref{lem25101}),   we have $\mathbb{E}_{\boldsymbol{\eta}} |\widetilde{g}_{N,\boldsymbol{\eta}}^{(N)}(\boldsymbol{w}) - a_{\boldsymbol{w}}|^2 \leq \frac{  1}{N^d}$ for $\boldsymbol{w}\in U$. Applying Parseval’s identity, we have 
\begin{equation*}
\begin{split}
\mathbb{E}_{\boldsymbol{\eta}} \mathbb{E}_{\boldsymbol{z}}&|g_{N,\boldsymbol{\eta}}(\boldsymbol{z}/N) - g(\boldsymbol{z}/N)|^2 = \mathbb{E}_{\boldsymbol{\eta}}\sum_{\boldsymbol{w}\in U} |\widetilde{g}_{N,\boldsymbol{\eta}}^{(N)}(\boldsymbol{w}) - a_{\boldsymbol{w}}|^2 + \mathbb{E}_{\boldsymbol{\eta}}\sum_{\boldsymbol{w} \in \mathbb{Z}_{N}^{d}\backslash U }|\widetilde{g}_{N,\boldsymbol{\eta}}^{(N)}(\boldsymbol{w})|^2 \\
&\le \frac{ K}{N^d}+\mathbb{E}_{\boldsymbol{\eta}}\sum_{\boldsymbol{w}\in \mathbb{Z}_{N}^{d} }|\widetilde{g}_{N,\boldsymbol{\eta}}^{(N)}(\boldsymbol{w})|^2-\mathbb{E}_{\boldsymbol{\eta}}\sum_{\boldsymbol{w}\in U }|\widetilde{g}_{N,\boldsymbol{\eta}}^{(N)}(\boldsymbol{w})|^2\\
&\le \frac{ K}{N^d}+\mathbb{E}_{\boldsymbol{\eta}}\left(\frac{1}{N^d}\sum_{\boldsymbol{z}\in\mathbb{Z}_{N}^d}|g_{N,\boldsymbol{\eta}}(\boldsymbol{z}/N)|^{2}\right)-\|g\|_{L^{2}([0,1]^d)}+\mathbb{E}_{\boldsymbol{\eta}}\sum_{\boldsymbol{w}\in U }\left||a_{\boldsymbol{w}}|^{2}-|\widetilde{g}_{N,\boldsymbol{\eta}}^{(N)}(\boldsymbol{w})|^2\right|\\
&\le \frac{ K}{N^d}+0+2   \mathbb{E}_{\boldsymbol{\eta}}\sum_{\boldsymbol{w}\in U }\left||a_{\boldsymbol{w}}|-|\widetilde{g}_{N,\boldsymbol{\eta}}^{(N)}(\boldsymbol{w})|\right|\le \frac{ K}{N^d}+2   \mathbb{E}_{\boldsymbol{\eta}}\sum_{\boldsymbol{w}\in U }\left|a_{\boldsymbol{w}}-\widetilde{g}_{N,\boldsymbol{\eta}}^{(N)}(\boldsymbol{w})\right|\le \frac{3K}{N^{d/2}}.
\end{split}
\end{equation*}
 Applying Markov's inequality and the above inequalities , we obtain the following two inequalities (for $\mathbb{E}_{\boldsymbol{\eta}} \mathbb{E}_{\boldsymbol{z}}$ and $\mathbb{E}_{\boldsymbol{\eta}}$, respectively):
\begin{equation}\label{lem250102}
\mathbb{P}\left\{ |g_{N,\boldsymbol{\eta}}(\boldsymbol{z}/N) - g(\boldsymbol{z}/N)|^{2}\le \frac{6K }{\sigma_{2} N^{\frac{d}{2}}}\right\}\ge1-\frac{\sigma_{2}N^{\frac{d}{2}}\mathbb{E}_{\boldsymbol{\eta}} \mathbb{E}_{\boldsymbol{z}}|g_{N,\boldsymbol{\eta}}(\boldsymbol{z}/N) - g(\boldsymbol{z}/N)|^2}{6K}\ge 1-\frac{\sigma_{2}}{2},\end{equation} 
and \begin{equation} \label{lem250102chai}   
\mathbb{P}\left\{ \mathbb{E}_{\boldsymbol{z}} |g_{N,\boldsymbol{\eta}}(\boldsymbol{z}/N) - g(\boldsymbol{z}/N)|^{2}\le \frac{6K }{\sigma_{2} N^{\frac{d}{2}}}\right\}\ge1-\frac{\sigma_{2}N^{\frac{d}{2}}\mathbb{E}_{\boldsymbol{\eta}} \left(\mathbb{E}_{\boldsymbol{z}}|g_{N,\boldsymbol{\eta}}(\boldsymbol{z}/N) - g(\boldsymbol{z}/N)|^2\right)}{6K}\ge 1-\frac{\sigma_{2}}{2}.
\end{equation} 
 By (Eq.~\eqref{lem2500}), (Eq.~\eqref{lem2501}), and (Eq.~\eqref{lem250102chai}),  
\begin{equation}
\begin{split}
&\frac{1}{N^d}\sum_{\boldsymbol{z}\in \mathbb{Z}_{N}^d}\left|R_{2}(\boldsymbol{z}/N)\right|^2\le \frac{2}{N^d}\sum_{\boldsymbol{z}\in \mathbb{Z}_{N}^d}|R_{N,\boldsymbol{\eta}}(\boldsymbol{z}/N)|^2+\frac{4}{N^d}\sum_{\boldsymbol{z}\in \mathbb{Z}_{N}^d}\left|g_{N,\boldsymbol{\eta}}(\boldsymbol{z}/N)-g(\boldsymbol{z}/N)\right|^2 \\&+\frac{4}{N^d}\sum_{\boldsymbol{z}\in \mathbb{Z}_{N}^d}\left|\widetilde{g}_{N,\boldsymbol{\eta}}^{(N)}(\boldsymbol{0})+\widetilde{R}_{N,\boldsymbol{\eta}}^{(N)}(\boldsymbol{0})\right|^2\le \frac{4}{\sigma_{2}}\|R\|_{L^{2}([0,1]^d)}^{2}+\Theta(\epsilon^2)+\frac{24K}{\sigma_{2} N^{d/2}}
\end{split}
\end{equation} 
holds with probability at least $1-\sigma_2-\frac{1}{\epsilon^{2} N^d}$. 
\end{proof}

To establish the theoretical analysis framework for our algorithm, we meticulously design a function space that is comprised of the function $f$ specifically defined in \eqref{NDf}:
\begin{equation}\label{fcsp}
\begin{split}
&F_{K,\mathcal{M},\lambda_{1},\lambda_{2}}:=\bigg\{f(\boldsymbol{x})=a_{\boldsymbol{0}}+g(\boldsymbol{x})+R(\boldsymbol{x}),\boldsymbol{x}\in[0,1]^d\\&\bigg|g(\boldsymbol{x})=\sum_{\boldsymbol{\xi}\in  U}a_{\boldsymbol{\xi}}\exp(2\pi i \boldsymbol{\boldsymbol{\xi}}\cdot \boldsymbol{x}), U\subseteq \{\boldsymbol{w}=(w_1,w_2,\dots,w_d)\in \mathbb{Z}^d \backslash\{\boldsymbol{0}\}  | \min_{j\in\supp(\boldsymbol{w})}|w_{j}|\le\mathcal{M}\}, \\
&
\sum_{\boldsymbol{\xi}\in  U}|a_{\boldsymbol{\xi}}|\le \lambda_{1},\#U=K, R(\boldsymbol{x})=\sum_{\boldsymbol{w}\in \mathbb{Z}^d \backslash (U\cup \{\boldsymbol{0}\}) } b_{\boldsymbol{w}}\exp(2\pi i \boldsymbol{w}\cdot \boldsymbol{x}),\|R\|_{L^{2}([0,1]^d)}^2\le \lambda_{2}\bigg\}.
\end{split}
\end{equation}
It is obvious that $\operatorname{INT}(f) =\hat{f}(\boldsymbol{0}) =a_{\boldsymbol{0}}$ for any $f\in F_{K,\mathcal{M},\lambda_{1},\lambda_{2}}$.
\begin{theorem}\label{T1}Let
$f\in F_{K,\mathcal{M},\lambda_{1},\lambda_{2}} $ with $\|f\|_{L^{2}([0,1]^d)}\le 1$. Given $0<\sigma,\sigma_2\le 1/10$, $n\ge K+1$, $0<\epsilon<\sigma/n$, we take  $B=\lceil   4n/\sigma  \rceil$, $r=B \sqrt{ \log (1/\epsilon)}$, $L=\lceil r\sqrt{2\log (1/\epsilon)}\rceil$, and let $N=\max\{\Theta(1/(\sqrt{\sigma}\epsilon)^{2/d}),\Theta((K/(\sigma_{2}\epsilon^2))^{2/d}),\mathcal{M}+1,3L\}$ be a prime number. Then it holds that over the randomness of $\boldsymbol{H}_{N}$, ${\boldsymbol{z}}$, and $\boldsymbol{\eta}$, 
 \begin{equation}\label{T1E1}
\mathbb{P}\left\{\left|I(f_{N,\boldsymbol{\eta}})_{\boldsymbol{H}_{N},L,r,\boldsymbol{z}}-\operatorname{INT}(f) \right |\le  \Theta\left((\lambda_{1}+1) \epsilon+ \sqrt{\frac{\lambda_{2}\log(1/\epsilon)}{L\sigma \sigma_2}}
\right)\right\}>1-\sigma-\sigma_2.
\end{equation}
Taking the median of outputs $\{I(f_{N,\boldsymbol{\eta}})_{\boldsymbol{H}_{N_j},L,r,\boldsymbol{z}_{j}}\}_{j=1}^{t}$,
which are obtained by independently randomizing $\boldsymbol{H}_{N}$, ${\boldsymbol{z}}$, and $\boldsymbol{\eta}$ for $t$ times, where $t$ is an odd number greater than or equal to $7$, we have  

\begin{equation}\label{T1E2}
\begin{split}
&~~\mathbb{P}\left\{\left|{\rm median} \left\{I(f_{N,\boldsymbol{\eta}_{j}})_{\boldsymbol{H}_{N_j},L,r,\boldsymbol{z}_{j}}\right\}_{j=1}^{t}-\operatorname{INT}(f)\right| \le \Theta\left( (\lambda_{1}+1) \epsilon+ \sqrt{\frac{\lambda_{2}\log(1/\epsilon)}{L\sigma\sigma_2}}\right)\right\}\\ &>1-(2\sqrt{\sigma+\sigma_2})^t.
\end{split}
\end{equation}
Furthermore, the RMSE of the presented integration method is bounded by
\begin{equation}\label{253}
\begin{split}
&\left(\mathbb{E}_{\{\boldsymbol{H}_{N_j},\boldsymbol{z}_{j},\boldsymbol{\eta}_{j}\}_{j=1}^t}\left|{\rm median}\left\{I(f_{N,\boldsymbol{\eta}_{j}})_{\boldsymbol{H}_{N_j},L,r,\boldsymbol{z}_{j}}\right\}_{j=1}^{t}-\operatorname{INT}(f)\right|^{2}\right)^{1/2}\\ &\le \Theta\left((\lambda_{1}+1) \epsilon+ \sqrt{\lambda_{2}\log(1/\epsilon)/(L\sigma \sigma_2)}+(2\sqrt{\sigma+\sigma_2})^t  + \frac{1}{5^{t/2}}\right).~~~~~~~~~~~~~~
\end{split}
\end{equation}
\end{theorem}

\begin{proof}
For any function $f\in F_{K,\mathcal{M},\lambda_{1},\lambda_{2}}$, $f$ can be can be expressed as
\begin{equation}\label{253f}
f(\boldsymbol{x})=a_{\boldsymbol{0}}+g(\boldsymbol{x})+R(\boldsymbol{x}),\end{equation}
with the conditions
$$
\sum_{\boldsymbol{\xi}\in  U}|a_{\boldsymbol{\xi}}|\le \lambda_{1},\operatorname{INT}(R)=0, ~\text{and}~\|R\|_{L^{2}([0,1]^d)}^2\le \lambda_{2}.
$$
As presented in (Eq.~\eqref{fnyt}), we have
\begin{equation}\label{zb}f_{N,\boldsymbol{\eta}}(\boldsymbol{z}/N)=a_{\boldsymbol{0}}^{*}+g(\boldsymbol{z}/N)+R_{2}(\boldsymbol{z}/N), \quad \boldsymbol{z}\in\mathbb{Z}_{N}^d. 
\end{equation}
To bound the integration error, by (Eq.~\eqref{1021}) and that $a_{\boldsymbol{0}}={\rm INT}(f)$, it suffices to estimate the following three terms:
\begin{equation}\label{shilia}
    \left|\sum_{|l|\le L }   a_{\boldsymbol{0}}^{*} G_{r,l}-a_{\boldsymbol{0}}\right|,~\left| \sum_{|l|\le L }g\left(\left\{\frac{\boldsymbol{z}-l\boldsymbol{H}_{N}}{N}\right\}\right) G_{r,l}\right|,~\text{and}~\left|\sum_{|l|\le L}R_{2}\left(\left\{\frac{\boldsymbol{z}-l\boldsymbol{H}_{N}}{N}\right\}\right) G_{r,l}\right|.
\end{equation}
By choosing a prime number 
 $N=\max\{\Theta(1/(\sqrt{\sigma}\epsilon)^{2/d}),\Theta((K/(\sigma_{2}\epsilon^2))^{2/d}),\mathcal{M}+1,3L\}$ and making use of Lemma \ref{3lem}, we have   \begin{equation}\label{aep}
    \mathbb{P}\left\{\left|   a_{\boldsymbol{0}}^{*} -a_{\boldsymbol{0}}\right|\le \Theta(\epsilon)\right\} \ge1-\sigma.\end{equation} 
Moreover, the following inequality holds with probability at least $1-\sigma-\sigma_2$,  
\begin{equation}\label{aiai}
    \frac{1}{N^d}\sum_{\boldsymbol{z}\in \mathbb{Z}_{N}^d}\left|R_{2}(\boldsymbol{z}/N)\right|^{2}=\sum_{\boldsymbol{w}\in \mathbb{Z}_{N}^d} \left| \widetilde{R_2}^{(N)}(\boldsymbol{w})\right|^{2} \le \frac{\|R\|_{L^{2}([0,1]^d)}^{2}}{\sigma_{2}}+ \Theta(\epsilon^2) \le  \frac{4\lambda_{2}}{\sigma_{2}}+ \Theta(\epsilon^2).\end{equation} 
 The application of (Eq.~(\ref{sshuoming})) and (Eq.~(\ref{shuoming2})) (recall the condition $N\ge \mathcal{M}+1$) yields that: 
$$\widetilde{R}_{2}^{(N)}(\boldsymbol{0})=0,~\widetilde{g}^{(N)}(\boldsymbol{0})=0,~\text{and}~\widetilde{f}_{N,\boldsymbol{\eta}}^{(N)}(\boldsymbol{0})=a_{\boldsymbol{0}}^{*}+\widetilde{g}^{(N)}(\boldsymbol{0})+\widetilde{R}_{2}^{(N)}(\boldsymbol{0})=a_{\boldsymbol{0}}^{*}.
$$
From Lemma \ref{seeea} and (Eq.~(\ref{aep})), we deduce
\begin{equation}\label{fanlei}
\mathbb{P}\left\{\left|\sum_{|l|\le L }   a_{\boldsymbol{0}}^{*} G_{r,l}-a_{\boldsymbol{0}}\right|\le \Theta(\epsilon)\right\} \ge1-\sigma.\end{equation}
We will analyze the last two terms  of (Eq.~\eqref{shilia}) separately. Let $n\ge K+1$, $B= \lceil 4n/\sigma \rceil$, $r=B \sqrt{\log (1 / \epsilon)}$, and $L=r\sqrt{2\log (1 / \epsilon)}$. By Lemma \ref{seeeaa}, for $\boldsymbol{w}^{(j)}\in U$, 
$$
\mathbb{P}\left\{ \boldsymbol{w}^{(j)}\in U_{\boldsymbol{H}_{N},B}\right\}\le\sigma/n,~1 \le j\le K,
$$ 
which leads to 
$$ 
\mathbb{P}\left\{\boldsymbol{w}^{(j)}\notin U_{\boldsymbol{H}_{N},B}~\text{holds}~ \text{for}~\text{every}~ j=1,2,\dots,K\right\}\ge 1-\sigma.
 $$
Furthermore, by Lemma \ref{seeea}, it holds
\begin{equation}\label{zy1}
  \mathbb{P}\left\{\sum_{|l|\le L }  \left| g\left(\left\{\frac{\boldsymbol{z}-l\boldsymbol{H}_{N}}{N}\right\}\right) G_{r,l}\right|\le 10\lambda_{1}\epsilon~\text{holds}~ \text{for}~\text{every}~ \boldsymbol{z}\in\mathbb{Z}_{N}^d\right\}\ge 1-\sigma.
\end{equation}
Next, we analyze $R_{2}$.
For each $\boldsymbol{\xi} \in \mathbb{Z}_{N}^{d} \backslash\{\boldsymbol{0}\}$, we let $\widetilde{R_2}^{(N)}(\boldsymbol{\xi})=c_{\boldsymbol{\xi}}$ and denote
\begin{equation}\label{tdingyi}
 t_{\boldsymbol{\xi}}:=\frac{1}{N^d}\sum_{\boldsymbol{z}\in\mathbb{Z}_{N}^d }\left| \sum_{|l|\le L}\exp \left(2\pi i \boldsymbol{\xi}\cdot\left\{\frac{\boldsymbol{z}-l\boldsymbol{H}_{N}}{N}\right\}\right) G_{r,l}\right|^{2}. \end{equation}
The application of Lemma \ref{seeea} leads to $\sum_{ |l|\le L }G_{r,l} \le 1+10\epsilon\le 2$. Therefore,  
\begin{equation} \label{unb0}
t_{\boldsymbol{\xi}}\le \frac{1}{N^d}\sum_{\boldsymbol{Z}\in \mathbb{Z}_{N}^{d} }\Big| \sum_{|l|\le L}G_{r,l}\Big|^{2}\le \Theta(1).
\end{equation} 
In addition, by using  Lemma \ref{seeeaa}, we derive 
\begin{equation} \label{unb}  
\mathbb{P}\Big\{\boldsymbol{\xi}\in U_{\boldsymbol{H}_N,B} \Big\}\le \Theta(1/B).
\end{equation}   
The combination of (Eq.~\eqref{unb0}) and (Eq.~\eqref{unb}) yields
$$\mathbb{E}_{\boldsymbol{H}_{N}}\Big(t_{\boldsymbol{\xi}}|\boldsymbol{\xi}\in U_{\boldsymbol{H}_N,B}\Big)\mathbb{P}\Big\{\boldsymbol{\xi}\in U_{\boldsymbol{H}_N,B} \Big\}\le \Theta\bigg(\frac{1}{B}\bigg).$$
On the other hand, by using (Eq.~(\ref{l23e10})), we have 
\begin{equation}\label{whm}  
\mathbb{E}_{\boldsymbol{H}_{N}}\big(t_{\boldsymbol{\xi}}|\boldsymbol{\xi}\notin U_{\boldsymbol{H}_N,B}\big)\le \Theta(\epsilon^2 ).
\end{equation}
Therefore, it follows that for $~\boldsymbol{H}_{N}{\sim} \operatorname{unif}\mathbb{Z}^{d}\cap [1,N)^{d}$,
\begin{equation}\label{tere}
\begin{split}
\mathbb{E}_{\boldsymbol{H}_{N}}t_{\boldsymbol{\xi}}&= \mathbb{E}_{\boldsymbol{H}_{N}}\big(t_{\boldsymbol{\xi}}|\boldsymbol{\xi}\in U_{\boldsymbol{H}_N,B}\big)\mathbb{P}\left\{\boldsymbol{\xi}\in U_{\boldsymbol{H}_N,B} \right\}+ \mathbb{E}_{\boldsymbol{H}_{N}}\big(t_{\boldsymbol{\xi}}|\boldsymbol{\xi}\notin U_{\boldsymbol{H}_N,B}\big)\mathbb{P}\left\{\boldsymbol{\xi}\notin U_{\boldsymbol{H}_N,B}\right\}\\&\le \Theta(1/B+\epsilon^2 ),
\end{split}
\end{equation}
Upon observing that for every distinct $\boldsymbol{\xi},\boldsymbol{\eta}\in \mathbb{Z}_{N}^d$, 
\begin{equation}\label{zj}
\begin{split}
\sum_{\boldsymbol{z}\in\mathbb{Z}_{N}^d }\left( \sum_{|l|\le L}\exp \left(2\pi i\boldsymbol{\xi}\cdot\left\{\frac{\boldsymbol{z}-l\boldsymbol{H}_{N}}{N}\right\}\right) G_{r,l}\right)\left( \sum_{|l|\le L}\exp \left(-2\pi i\boldsymbol{\eta}\cdot\left\{\frac{\boldsymbol{z}-l\boldsymbol{H}_{N}}{N}\right\}\right) G_{r,l}\right)=0,
\end{split}
\end{equation}
and $\widetilde{R_2}^{(N)}(\boldsymbol{\xi})=c_{\boldsymbol{\xi}}$, by combining (Eq.~\eqref{aiai}) and (Eq.~\eqref{tere}), we immediately arrive at the conclusion that with probability at least $1-\sigma-\sigma_2$ over the randomness of $\boldsymbol{\eta}$, 
\begin{equation}\label{simil} 
\begin{split}
\mathbb{E}_{\boldsymbol{H}_{N}}\mathbb{E}_{\boldsymbol{z}}\bigg| \sum_{|l|\le L}R_{2}\Big(\left\{\frac{\boldsymbol{z}-l\boldsymbol{H}_{N}}{N}\right\}\Big) G_{r,l}\bigg|^{2}&=\mathbb{E}_{\boldsymbol{H}_{N}}\sum_{\boldsymbol{\xi}\in  \mathbb{Z}_{N}^d\backslash\{0\} }\left| \widetilde{R_2}^{(N)}(\boldsymbol{\xi})\right|^{2}t_{\boldsymbol{\xi}}\\&\le   \Theta\left(\left(\frac{1}{B}+\epsilon^2\right)\left(\frac{\lambda_2}{\sigma_2}+\epsilon^2\right)\right) ,
\end{split}
\end{equation}  
 where $\boldsymbol{z}\sim\mbox{unif}\ \mathbb{Z}_{N}^d$. 
 
Recalling the conditions $B= \lceil 4n/\sigma \rceil$ and $\epsilon<1/n$, we obtain the following inequality by Markov's inequality, 
\begin{equation}\label{zy2}
\mathbb{P}\left\{ \left|\sum_{|l|\le L}R_{2}\left(\left\{\frac{\boldsymbol{z}-l\boldsymbol{H}_{N}}{N}\right\}\right) G_{r,l}\right|^{2}  \le \Theta\left(\frac{\lambda_{2}}{n \sigma_{2} } + \frac{\epsilon^2}{n} \right)\right\}\ge (1-\sigma)(1-\sigma-\sigma_2)\ge 1-2\sigma- \sigma_2,
\end{equation} 
where the probability is determined by the random variables $\boldsymbol{z}$, $H_N$, and $\boldsymbol{\eta}$.
Summarizing (Eq.~\eqref{zb}), (Eq.~\eqref{fanlei}), (Eq.~\eqref{zy1}), and (Eq.~\eqref{zy2}), and applying Markov's inequality, we obtain
$$\mathbb{P}\left\{ \left|\sum_{|l|\le L}f\left(\left\{\frac{\boldsymbol{z}-l\boldsymbol{H}_{N}}{N}\right\}\right) G_{r,l}-a_{\boldsymbol{0}}\right|\le \Theta( (\lambda_{1}+1) \epsilon+\sqrt{\lambda_{2}/(n  \sigma_{2})})\right\}\ge 1-4\sigma-\sigma_{2}.$$ 
Replacing $\sigma$ by  $\sigma/4$, we complete the proof of (Eq.~(\ref{T1E1})).

By randomly selecting $\boldsymbol{H}_{N}$, ${\boldsymbol{z}}$, and $\boldsymbol{\eta}$ to compute $I(f_{N,\boldsymbol{\eta}})_{\boldsymbol{H}_{N},L,r,\boldsymbol{z}}$ $t$ times, we obtain $t$ different results. Extracting the median of the obtained results, the estimation (\ref{T1E2}) can be naturally derived from Lemma \ref{Median}. Denote median$\{I(f_{N,\boldsymbol{\eta}_{j}})_{\boldsymbol{H}_{N_j},L,r,\boldsymbol{z}_{j}}\}_{j=1}^{t}$ by $Y$. Then
\begin{equation}
\begin{split}
&~~~~\mathbb{E}_{\{\boldsymbol{H}_{N_j},\boldsymbol{z}_{j},\boldsymbol{\eta}_{j}\}_{j=1}^t}\big|{\rm median}\{I(f_{N,\boldsymbol{\eta}_{j}})_{\boldsymbol{H}_{N_j},L,r,\boldsymbol{z}_{j}}\}_{j=1}^{t}-\operatorname{INT}(f)\big|^{2}\\
&\le\sum_{k\ge 0}\mathbb{P}\Big\{200k\le |Y|\le 200(k+1)\Big\}\mathbb{E}_{\{\boldsymbol{H}_{N_j},\boldsymbol{z}_{j},\boldsymbol{\eta}_{j}\}_{j=1}^t}\Big(\big|Y-\operatorname{INT}(f)\big|^2\Big|200k\le |Y|
\le 200(k+1)\bigg)\\&=:\sum_{k\ge 0}E_{k}.
\end{split}
\end{equation}
Combining $|\operatorname{INT}(f)|\le 1$ and (Eq.~(\ref{T1E2})), we obtain 
$$E_{0}\le \left( \Theta\left( (\lambda_{1}+1) \epsilon+ \sqrt{\lambda_{2}/(n  \sigma_{2})}\right)\left(1-(2\sqrt{\sigma+\sigma_{2}})^t\right)+(200+1)(2\sqrt{\sigma+\sigma_{2}})^{t}\right)^2.
$$
Notice
\begin{equation*}\begin{split}
&\mathbb{E}_{\boldsymbol{\eta}}\mathbb{E}_{\boldsymbol{z}}\left|I(f_{N,\boldsymbol{\eta}})_{\boldsymbol{H}_{N},L,r,\boldsymbol{z}}\right|\le \sum_{|l|\le L}G_{r,l}\mathbb{E}_{\boldsymbol{\eta}}\mathbb{E}_{\boldsymbol{z}} \left|f_{N,\boldsymbol{\eta}}\left(\left\{\frac{\boldsymbol{z}-l\boldsymbol{H}_{N}}{N}\right\}\right)\right|\\=&\Big(\sum_{|l|\le L}G_{r,l}\Big)\mathbb{E}_{\boldsymbol{\eta}}\mathbb{E}_{\boldsymbol{z}} \Big|f_{N,\boldsymbol{\eta}}(\boldsymbol{z}/N)\Big| =\big(\sum_{|l|\le L}G_{r,l}\big)
\|f\|_{L^1([0,1]^d)}
\le 2 \|f\|_{L^2\big([0,1]^d\big)}\le 2.
\end{split}
\end{equation*}
Thus, for $k\ge 1$, using Markov's inequality, we have $$\mathbb{P}\{ |I(f_{N,\boldsymbol{\eta}})_{\boldsymbol{H}_{N},L,r,\boldsymbol{z}}|\ge 200k\}\le \frac{1}{100k},$$
which follows that
$$\mathbb{P}\{ |Y|\ge 200k\}\le \binom{t}{\frac{t+1}{2}}\bigg(\frac{1}{100k}\bigg)^{\frac{t+1}{2}}\le \left(\frac{1}{5}\right)^{t}/k^{\frac{t+1}{2}}.$$
Therefore, 
$$E_{k}\le \big(200(k+2)\big)^{2}\left(\frac{1}{5}\right)^{t}/k^{\frac{t+1}{2}}, \quad \text{for}\quad k\ge1.$$
Given that $t\ge 7$, there exists an absolute constant $C$ such that $$\sum_{k\ge 1} (k+2)^2/k^{\frac{t+1}{2}}\le C.$$ 
Consequently, $$\sum_{k\ge 1}E_{k}\le \Theta \left(\frac{1}{5^t}\right),$$
and 
\begin{equation}
\begin{split}
&~~~~\left(\mathbb{E}_{\{\boldsymbol{H}_{N_j},\boldsymbol{z}_{j},\boldsymbol{\eta}_{j}\}_{j=1}^t}\Big|{\rm median}\big\{I(f_{N,\boldsymbol{\eta}})_{\boldsymbol{H}_{N_j},L,r,\boldsymbol{z}_{j}}\big\}_{j=1}^{t}-\operatorname{INT}(f)\Big|^{2}\right)^{1/2}\\&\le \left(\sum_{k\ge 0}E_{k}\right)^{1/2}\le \Theta\left( (\lambda_{1}+1) \epsilon+ \sqrt{\lambda_{2}/(n\sigma_2)}+(2\sqrt{\sigma+\sigma_2})^t+ \frac{1}{5^{t/2}} \right).
\end{split}
\end{equation}
This completes the proof.
\end{proof} 
\section{Integration in Isotropic Sobolev Spaces}
We are now prepared to develop an estimation for the periodic isotropic Sobolev spaces of order $s\ge0 $.
$$H^{s}(\mathcal{T}^d):=\left\{ f \in L^{2}(\mathcal{T}^d) \Bigg|\|f\|_{H^{s}(\mathcal{T}^d)}:=\sum_{\boldsymbol{w}=(w_{1},\dots,w_{d})\in \mathbb{Z}^d}\Big(1+ \big(\sum_{j=1}^{d}|2\pi w_{j}|\big)^{s}\Big)^2|\hat{f}(\boldsymbol{w})|^2< \infty\right\}.$$
\begin{theorem}\label{T2} Let $f\in H^{s}(\mathcal{T}^d)$ with $\|f\|_{H^{s}(\mathcal{T}^d)}\le 1$. Let $0<\sigma \le 1/10$, $n\ge 1$, $\epsilon<\sigma/n$, $B=\lceil4n/\sigma\rceil$, $r=B \sqrt{\log (1 / \epsilon)}$, $L=\lceil r\sqrt{2\log (1/\epsilon)}\rceil$ and $N=\max\{\Theta(1/(\sqrt{\sigma}\epsilon)^{2/d}),\Theta((n/\epsilon^2)^{2/d}),3L\}$ be a prime number. It holds with probability at least $1-\sigma-1/10$ over the randomness of $\boldsymbol{H}_{N}$, ${\boldsymbol{z}}$, and $\boldsymbol{\eta}$ that
\begin{equation}\label{T21}\left|I(f_{N,\boldsymbol{\eta}})_{\boldsymbol{H}_{N},L,r,\boldsymbol{z}}- \operatorname{INT}(f) \right |\le \Theta\left(\frac{\left(\log(1/\epsilon)\right)^{s/d+1/2}}{(L\sigma)^{s/d+1/2}}+\sqrt{L}\epsilon \right).\end{equation} 
Let ${\rm median}\{I(f_{N,\boldsymbol{\eta}_{j}})_{\boldsymbol{H}_{N_j},L,r,\boldsymbol{z}_{j}}\}_{j=1}^{t}$ be defined as in Theorem \ref{T1}. It holds with probability at least $1-(2\sqrt{\sigma+1/10})^t$ that
\begin{equation}\label{T22}\left|{\rm median}\left\{I(f_{N,\boldsymbol{\eta}_{j}})_{\boldsymbol{H}_{N_j},L,r,\boldsymbol{z}_{j}}\right\}_{j=1}^{t}-\operatorname{INT}(f)\right| \le \Theta\left(\frac{\left(\log(1/\epsilon)\right)^{s/d+1/2}}{(L\sigma)^{s/d+1/2}}+\sqrt{L}\epsilon\right). \end{equation}
Furthermore, the RMSE is bounded by
\begin{equation}\label{T23}
\begin{split}
\left(\mathbb{E}_{\{\boldsymbol{H}_{N_j},\boldsymbol{z}_{j},\boldsymbol{\eta}_{j}\}_{j=1}^t}\left|{\rm median}\left\{I(f_{N,\boldsymbol{\eta}})_{\boldsymbol{H}_{N_j},L,r,\boldsymbol{z}_{j}}\right\}_{j=1}^{t}-\operatorname{INT}(f)\right|^{2}\right)^{1/2} \\\le \Theta\left(\frac{\left(\log(1/\epsilon)\right)^{s/d+1/2}}{(L\sigma)^{s/d+1/2}}+\sqrt{L}\epsilon+  (2\sqrt{\sigma+1/10})^t\right).~~
\end{split}
\end{equation}
\end{theorem}
\begin{remark}\label{28888}
We use $M:=(2L+1)t$ to represent the total sample size employed by our algorithm. 
Substituting $\sigma=1/10$, $t=\log_{\sqrt{5}/2}(n^{s/d+1/2})$, and $\epsilon=1/n^{s/d+3/2}$ into (Eq.~(\ref{T23})), the RMSE error is bounded above by
\begin{equation}\label{sp1}
\Theta\left( (3+4s/d)^{4s/d+3} (\log M)^{{2s/d+1}} /M^{s/d+1/2}\right).
\end{equation} 
When the RMSE error is equal to $\epsilon$, it follows that the sample size $M$ is polynomial with respect to $1/\epsilon$. This demonstrates that the random algorithm has strong polynomial tractability.
The upper bound in (Eq.~(\ref{sp1})) (up to an absolute constant factor) also applies to integration in
the isotropic Sobolev space with compact support (see Corollary \ref{corollv} and Remark \ref{2112233}).
In contrast to \cite{krieg2017universal,ullrich2017Monte}, where the upper bounds for integration algorithms in this space are at least $2^{\Theta(d)}M^{-s/d-1/2}$, our bound outperforms these results when 
 $s\le \Theta(d)$ and $M\le 2^{2^{\Theta(d)}}$ (i.e., $\log^{\Theta(1)} M \le 2^{\Theta(d)}$). Additionally, in the absence of prior knowledge of the order $s$, setting $\epsilon=1/(n^{\log n})$ and $t=\log_{\sqrt{5}/2} (n^ {\log n})$ yields an algorithm universally applicable to all $s$.
\end{remark}
\begin{proof}
Let $n$ be a positive integer, and define $\beta_{n,d}=\frac{n^{1/d}}{2\pi}$. Decompose any function $f\in H^{s}(\mathcal{T}^d)$ as
$$f(x)=a_{\boldsymbol{0}}+g(\boldsymbol{x}) +R(\boldsymbol{x}),$$
where $$g(\boldsymbol{x})=\sum_{\boldsymbol{w}\in  \mathbb{Z}^{d}\cap(-\beta_{n,d},\beta_{n,d})^{d} \backslash\{\boldsymbol{0}\}  }a_{\boldsymbol{w}}\exp(2\pi i \boldsymbol{w}\cdot \boldsymbol{x}),$$
and $$R(\boldsymbol{x})=\sum_{\boldsymbol{w} \notin  \mathbb{Z}^{d}\cap(-\beta_{n,d},\beta_{n,d})^d }a_{\boldsymbol{w}}\exp(2\pi i \boldsymbol{w}\cdot \boldsymbol{x}).$$
Since $f\in H^{s}(\mathcal{T}^d)$,
it follows that $\|R\|_{L^{2}([0,1]^d)}\le \Theta(n^{-\frac{s}{d}}\|f\|_{ H^{s}})$. Additionally, 
$$W:=\#\Big\{\boldsymbol{w}\big| \boldsymbol{w}\in  \mathbb{Z}^{d}\cap(-\beta_{n,d},\beta_{n,d})^d \backslash\{\boldsymbol{0}\} \Big\}\le n-1.$$ 
By the Cauchy-Schwarz inequality, $$\sum_{\boldsymbol{w}\in W}|a_{\boldsymbol{w}}|\le \sqrt{n-1}(\sum_{\boldsymbol{w}\in W}|a_{\boldsymbol{w}}|^{2})^{1/2}\le\sqrt{n-1}.$$
Let $B= 4n/\sigma$. Theorem \ref{T2} follows immediately by letting $\sigma_2=1/10$ and $\mathcal{M}=\beta_{n,d}$ in Theorem \ref{T1}.
\end{proof}

Given a bounded measurable set $\Omega$ with volume $1$,  we denote the isotropic Sobolev spaces with compact support in $\Omega$ by $\mathring{H}^{s}(\Omega)$ (see the definition on page 1189 in \cite{ullrich2017Monte} for details). Specifically,
$$\mathring{H}^{s}(\Omega):=\left\{f\in \mathring{L}^{2}(\Omega)\}|D^{\alpha}f\in L^{2}(\mathbb{R}^d) \quad \text{for}\quad \alpha\in\mathbb{N}_{0}^d \quad \text{and}\quad|\alpha|_{1}\le s\right\}$$
equipped with the norm 
$$\|f\|_{\mathring{H}^{s}(\Omega)}:=\|f\|_{L^{2}(\Omega)}+\sum_{j=1}^{d}\left\|\frac{\partial^{s}f}{\partial x_{j}^{s}}\right\|_{L^2(\Omega)},$$
where $\mathring{L}^{2}(\Omega):=\left\{f\in L^{2}(\mathbb{R}^{d})|\supp(f)\subseteq\Omega\right\}$, $D^{\alpha}f$ denotes the $\alpha$-order weak partial derivative of a function $f$. 

For any $f\in\mathring{H}^{s}(\Omega)$, define the function
$$F(\boldsymbol{x})=\sum_{\boldsymbol{k}\in \mathbb{Z}^d}f(\boldsymbol{k}+\boldsymbol{x}).$$
It holds
$$\int_{\Omega}  f(\boldsymbol{x})d\boldsymbol{x}=\int_{ \mathcal{T}^d} F(\boldsymbol{x})d\boldsymbol{x},$$ and $$ \int_{\Omega} |D^{\alpha}f(\boldsymbol{x})|^{2}d\boldsymbol{x}=\int_{ \mathcal{T}^d}|D^{\alpha}F(\boldsymbol{x})|^{2}d\boldsymbol{x} \quad \text{for}\quad \text{every}\quad \alpha\in\mathbb{N}_{0}^{d}.$$
By the Parseval identity and Minkowski inequality, we have
 $$\|F\|_{H^{s}(\mathcal{T}^d)}\le\|f\|_{\mathring{H}^{s}(\Omega)}. $$ 
Consequently, the error bounds (\ref{T21}), (\ref{T22}), and (\ref{T23}) derived earlier remain valid for the integral $\operatorname{INT}(F)=\int_{\Omega} f(\boldsymbol{x})d\boldsymbol{x}$.

A relevant question concerns the number of samples required to determine the value of
$F(\boldsymbol{x})$. More precisely, based on (Eq.~(\ref{1021})), how many samples are needed to obtain the values of the set $\left\{ F_{N,\boldsymbol{\eta}}\left(\left\{\frac{\boldsymbol{z}-l\boldsymbol{H}_{N}}{N}\right\}\right)\right\}_{l=-L}^{L}$?  Define $$
\mathbb{J}^{\Omega}(\boldsymbol{x}):=\sum_{\boldsymbol{k}\in \mathbb{Z}^d}\mathbb{I}_{\Omega}(\boldsymbol{k}+\boldsymbol{x}),
$$
 With a certain probability, only $\Theta(L)$ samples are required. In fact,
\begin{equation*}
\begin{split}
~~~~~~2L+1=(2L+1)\mbox{volume}(\Omega)=(2L+1)\int_{\Omega}\mathbb{I}_{\Omega}(\boldsymbol{x})d\boldsymbol{x}~~~~~~~~~~~~~~~~~~~~~~~~~~~~~~~~
\\=(2L+1)\int_{[0,1]^d}\mathbb{J}^{\Omega}(\boldsymbol{x})d\boldsymbol{x}=\sum_{l=-L}^{L}\mathbb{E}_{\boldsymbol{\eta} }\mathbb{E}_{\boldsymbol{z}}\mathbb{J}_{N,\boldsymbol{\eta}}^{\Omega}\left(\frac{\boldsymbol{z}-l\boldsymbol{H}_{N}}{N}\right)~~~~~~~~~~~~~~~~~~~~~~~~~\\
=\mathbb{E}_{\boldsymbol{\eta} }\mathbb{E}_{\boldsymbol{z}}\sum_{l=-L}^{L}\#\left\{\boldsymbol{k}\big|\boldsymbol{k}\in \mathbb{Z}^d,  (\boldsymbol{z}-l\boldsymbol{H}_{N})/N+\boldsymbol{\eta}_{z}+\boldsymbol{k}\in\Omega\right\},~~~~~~~~~~~~~~~~~~~~~~~~~
\end{split}
\end{equation*}
where $\boldsymbol{z}{\sim} \operatorname{unif}\mathbb Z_{N}^{d} $, $-L\le l\le L$, and $\boldsymbol{\eta}_{\boldsymbol{z}}\stackrel{i i d}{\sim} \operatorname{unif}\,[0,1/N)^{d}$ (see (Eq.~\eqref{MCC})). By Markov's inequality,
$$
\mathbb{P}\left\{\sum_{l=-L}^{L}\#\left\{\boldsymbol{k}|\boldsymbol{k}\in \mathbb{Z}^d,  (\boldsymbol{z}-l\boldsymbol{H}_{N})/N+\boldsymbol{\eta}_{\boldsymbol{z}}+k\in\Omega\right\}\ge200L+100\right\}\le 1/100.
$$
In other words, with a probability of at least $99/100$, 
the number of samples required to obtain $\left\{ F_{N,\boldsymbol{\eta}}\left(\left\{\frac{\boldsymbol{z}-l\boldsymbol{H}_{N}}{N}\right\}\right)\right\}_{l=-L}^{L}$ is at most $200L+100$.

When searching the sets 
$$
V_l:=\Big\{(\boldsymbol{z}-l\boldsymbol{H}_{N})/N+\boldsymbol{\eta}_{\boldsymbol{z}}+\boldsymbol{k}|\boldsymbol{k}\in \mathbb{Z}^d,  (\boldsymbol{z}-l\boldsymbol{H}_{N})/N+\boldsymbol{\eta}_{\boldsymbol{z}}+\boldsymbol{k}\in\Omega\Big\},-L\le l\le L,
$$
no samples of the function $f$ are required. If the number of elements in the union $\bigcup_{-L\le l\le L} V_l $ exceeds $200L+100$, we promptly output the result (denoted as $I^{*}(F_{N,\boldsymbol{\eta}})_{\boldsymbol{H}_{N},L,r,\boldsymbol{z}}$) as zero. Otherwise, set $I^{*}(F_{N,\boldsymbol{\eta}})_{\boldsymbol{H}_{N},L,r,\boldsymbol{z}}=I(F_{N,\boldsymbol{\eta}})_{\boldsymbol{H}_{N},L,r,\boldsymbol{z}}$. 

This ensures that the sample complexity of our algorithm remains polynomial. Consequently, we have the following corollary.

\begin{corollary}\label{corollv}
Let $f\in\mathring{H}^{s}(\Omega)$ with $\|f\|_{\mathring{H}^{s}(\Omega)}\le 1$. Let $F(\boldsymbol{x})=\sum_{\boldsymbol{k}\in \mathbb{Z}^d}f(\boldsymbol{k}+\boldsymbol{x})$. Assume the same parameter settings as those in Theorem \ref{T2}. Then it holds with probability at least  $1-\sigma-1/10-1/100$ over the randomness of $\boldsymbol{H}_{N}$, ${\boldsymbol{z}}$, and $\boldsymbol{\eta}$ that
\begin{equation*}\bigg|I^{*}(F_{N,\boldsymbol{\eta}})_{\boldsymbol{H}_{N},L,r,\boldsymbol{z}}- \operatorname{INT}(F) \bigg |\le \Theta\left(\frac{1}{(L\delta/\log(1/\epsilon))^{s/d+1/2}}+\sqrt{L}\epsilon \right).\end{equation*} 
By obtaining  ${\rm median}\{I^{*}(F_{N,\boldsymbol{\eta}_{j}})_{\boldsymbol{H}_{N_j},L,r,\boldsymbol{z}_{j}}\}_{j=1}^{t}$ as in Theorem \ref{T1}, we have with probability at least $1-(2\sqrt{\sigma+1/10+1/100})^t$ that
\begin{equation*}\big|{\rm median}\{I^{*}(F_{N,\boldsymbol{\eta}_{j}})_{\boldsymbol{H}_{N_j},L,r,\boldsymbol{z}_{j}}\}_{j=1}^{t}-\operatorname{INT}(F)\big| \le \Theta\left(\frac{1}{(L\delta/\log(1/\epsilon))^{s/d+1/2}}+\sqrt{L}\epsilon\right).\end{equation*}
Furthermore, the RMSE is bounded by
\begin{equation*}
\begin{split}
&\left(\mathbb{E}_{\{\boldsymbol{H}_{N_j},\boldsymbol{z}_{j},\boldsymbol{\eta}_{j}\}_{j=1}^t}\big|{\rm median}\{I^{*}(F_{N,\boldsymbol{\eta}})_{\boldsymbol{H}_{N_j},L,r,\boldsymbol{z}_{j}}\}_{j=1}^{t}-\operatorname{INT}(F)\big|^{2}\right)^{1/2} \\&\le \Theta\left(\frac{1}{(L\delta/\log(1/\epsilon))^{s/d+1/2}}+\sqrt{L}\epsilon+  (2\sqrt{\sigma+1/10+1/100})^t\right).~~~~~~~~~~~~~~~~
\end{split}
\end{equation*}
\end{corollary}
\begin{remark}\label{2112233} 
We adopt the definition of the isotropic Sobolev space from \cite{ullrich2017Monte}, which differs subtly from that in \cite{krieg2017universal}. Nevertheless, the theorem remains applicable to the compactly supported isotropic Sobolev space as defined in \cite{krieg2017universal}.
\end{remark}
\begin{remark}
The analysis of the algorithm's sample complexity and polynomial tractability is parallel to that in Theorem \ref{T2} and Remark \ref{28888}.
\end{remark}
\section{Integration of Analytic and Wiener-type Functions}
In this section, we estimate the integral error for analytic and Wiener-type functions. We derive semi-exponential convergence rates for the numerical integration of analytic functions in the asymptotic sense. 

To this end, we first present a lemma to estimate the coefficients of higher-order derivatives of composite functions. Faà di Bruno's formula \cite{roman1980formula} states that the 
$k$-th order derivative of a composite function
$g\circ f$ admits the following representation: 
 \begin{equation}\label{faa}    
(g\circ f)^{(k)}(x)=\sum_{p=1}^{k}\left(\sum_{\substack{\sum_{j=1}^k j p_j = k \\ \sum_{j=1}^k p_j =p}}a(p,p_1,\dots,p_{k})g^{(p)}(f(x))\prod_{j=1}^{k} (f^{(j)}(x))^{p_j}\right),
\end{equation}
where $a(p,p_1,\dots,p_{k})$ are non-negative constants dependent only on $p,p_1,\dots,p_{k}$. 
\begin{lemma}\label{bucc}
Let $k\in \mathbb{N}$ and $a(p,p_1,\dots,p_{k})$ be defined as in (Eq.~\ref{faa}). It holds
$$\sum_{p=1}^{k}\sum_{\substack{\sum_{j=1}^k j p_j = k \\ \sum_{j=1}^k p_j = p}}a(p,p_1,\dots,p_{k})\le 4^{k}.$$
\end{lemma}
\begin{proof}
For $1\le q\le k$, consider $ g_q(x)=x^{q}/q!$ and $f(x)=e^x$. By Faà di Bruno's formula, 
\begin{equation*}\begin{split}
q^{q}/q!=(g_q\circ f)^{(q)}\big|_{x=0}&=\sum_{p=1}^{q}\biggl(\sum_{\substack{\sum_{j=1}^k j p_j = k \\ \sum_{j=1}^k p_j = p}}a(p,p_1,\dots,p_{k})\frac{1}{(q-p)!} \biggr)\\&\ge\sum_{\substack{\sum_{j=1}^k j p_j = k \\ \sum_{j=1}^k p_j = q}}a(p,p_1,\dots,p_{k})  .\end{split}   \end{equation*}
Summing over $q$ from $1$ to $k$ and using Stirling’s formula, we obtain:
$$\sum_{q=1}^{k}\sum_{\substack{\sum_{j=1}^k j p_j = k\\ \sum_{j=1}^k p_j = q}}a(p,p_1,\dots,p_{k})\le \sum_{q=1}^{k}q^{q}/q!\le k(ek)^{k}/k^{k} \le4^{k}.$$
\end{proof} 
Now, we present the main idea of the error analysis.
For any analytic function $f$ defined on $[0,1]^d$, we first use the trick of change of variable (see \cite{ullrich2017Monte,nguyen2017change,temlyakov2003cubature}) to construct a smooth function $Tf$,  which vanishes on the boundary of the hypercube $[0,1]^d$ and satisfies the integral identity$\int_{[0,1]^d}Tf(\boldsymbol{t})dt=\int_{[0,1]^d}f(\boldsymbol{t})dt.$ 
Specifically, we utilize the function introduced in \cite{nguyen2017change}:
\begin{equation}\label{ta1e99}
\psi(y) = 
\begin{cases} 
\frac{\int_{0}^{y} \exp(-1/(\xi(1-\xi))) \mathrm{~d} \xi }{ \int_{0}^{1} \exp(-1/(\xi(1-\xi))) \mathrm{~d} \xi} & \text{if } y\in[0,1],\\
1 & \text{if } y>1,\\
0 & \text{if } y<0.\\
\end{cases}
\end{equation}
For $\boldsymbol{t}=(t_1,t_2,\dots,t_d)\in \mathbb{R}^d$, let $f_{\psi}(\boldsymbol{t})=f(\psi(t_1),\psi(t_2),\dots,\psi(t_d)),$ and define $$Tf(\boldsymbol{t}):=\bigg|\prod_{j=1}^{d}\psi^{'}(t_{j})\bigg|f_{\psi}(\boldsymbol{t})=\left(\prod_{j=1}^{d}\psi^{'}(t_{j})\right)f_{\psi}(\boldsymbol{t}),$$
It follow that $\int_{[0,1]^d} f(\boldsymbol{t})dt=\int_{[0,1]^d} Tf(\boldsymbol{t})dt$. 
 For arbitrary
 analytic function $f$ defined on $[0,1]^d$,  applying the results proposed in \cite{pan2024super} (see page 2273), we know that there exist constants $C_{f,1}>1$ and $C_{f,2}>1$ depending only on $f$, such that 
\begin{equation*} 
\left|\frac{\partial^{J_{1}+\cdots+J_{d}} f}{\partial x_{1}^{J_{1}} \ldots \partial x_{d}^{J_{d}}}(\boldsymbol{t})\right| \le C_{f,1}(C_{f,2})^{J}J! \quad \text{for} \quad \boldsymbol{t}\in [0,1]^d,
\end{equation*}
where $\{J_{1},J_{2},\dots,J_{d}\}\subseteq \mathbb{N}_{0}$ and $J=J_{1}+\cdots+J_{d}$.

 We consider the more general case, i.e., Theorem \ref{ta1} applies to the functions that satisfy the  following inequality
\begin{equation}\label{ta1e1}
\left\|\frac{\partial^{J_{1}+\cdots+J_{d}} f}{\partial x_{1}^{J_{1}} \ldots \partial x_{d}^{J_{d}}}\right\|_{L^{1}[0,1]^d} \le C_{f,1}(C_{f,2})^{J}J! \quad \text{for} \quad \boldsymbol{t}\in [0,1]^d,
\end{equation}
where $C_{f,1}>1$ and $C_{f,2}>1$ are some constants depending only on $f$, $\{J_{1},J_{2},\dots,J_{d}\}\subseteq \mathbb{N}_{0}$ and $J=J_{1}+\cdots+J_{d}$. 
Both analytic functions defined on $[0,1]^d$ and the infinitely differentiable functions defined in \cite{suzuki2017super,dick2017construction} satisfy (Eq.~\eqref{ta1e1}).  By demonstrating that the Fourier coefficients of $Tf$ exhibit semi-exponential decay and using Theorem \ref{T1}, we establish the following result.

\begin{theorem}\label{ta1} 
Assume that $f$ satisfies (Eq.~\eqref{ta1e1}). Let the parameters satisfy $0<\sigma \le 1/10$, $n>3$, $0<\epsilon <\sigma/n$, $B=\lceil\frac{4n}{\sigma}\rceil$, $r=B \sqrt{\log (\frac{1} { \epsilon})}$, and $L=\lceil r\sqrt{2\log (\frac{1} { \epsilon})}\rceil$. Additionally, let $N$ be a prime number such that $N=\max\left\{\Theta\left(\frac{1}{(\sqrt{\sigma}\epsilon)^{2/d}}\right),\Theta\left((n/(\ln(L)\epsilon^2))^{2/d}\right),3L\right\}$. Then with respect to the randomness of 
$\boldsymbol{H}_{N}$, ${\boldsymbol{z}}$, and $\boldsymbol{\eta}$, the following inequality holds with a probability of at least  $1-1/(\ln L)$: 
\begin{equation}\label{T991}\bigg|I(Tf_{N,\boldsymbol{\eta}})_{\boldsymbol{H}_{N},L,r,\boldsymbol{z}}- \operatorname{INT}(f) \bigg |\le 
C_{f,0}\left(\exp \left(-\frac{ L^{{C_{*}/d}}}{C_{f,*}}\right) +\sqrt{L}\epsilon\right),
\end{equation} 
where $C_{f,0}$ represents positive constant dependent on both the function $f$ and the dimension $d$, $C_{f,*}$ denotes positive constant dependent on $f$,
and $C_{*}$ is a positive absolute constant.
Let ${\rm median}\{I(Tf_{N,\boldsymbol{\eta}_{j}})_{\boldsymbol{H}_{N_j},L,r,\boldsymbol{z}_{j}}\}_{j=1}^{t}$ be defined as in Theorem \ref{T1}. Then the following inequality holds with a probability of at least $1-(2\sqrt{1/(\ln L)})^t$
\begin{equation}\label{T92}\left|{\rm median}\{I(Tf_{N,\boldsymbol{\eta}_{j}})_{\boldsymbol{H}_{N_j},L,r,\boldsymbol{z}_{j}}\}_{j=1}^{t}-\operatorname{INT}(f)\right| \le C_{f,0}\left(\exp\left(-\frac{ L^{{C_{*}/d}}}{C_{f,*}}\right)+\sqrt{L}\epsilon\right). \end{equation}
Furthermore, 
\begin{equation}\label{T993}
\begin{split}
\bigg(\mathbb{E}_{\{\boldsymbol{H}_{N_j},\boldsymbol{z}_{j},\boldsymbol{\eta}_{j}\}_{j=1}^t}\big|{\rm median}\{I(Tf_{N,\boldsymbol{\eta}})_{\boldsymbol{H}_{N_j},L,r,\boldsymbol{z}_{j}}\}_{j=1}^{t}-\operatorname{INT}(f)\big|^{2}\bigg)^{1/2}\\ \le C_{f,0}\left(\exp\left(-\frac{ L^{{C_{*}/d}}}{C_{f,*}}\right)+\sqrt{L}\epsilon+\left(2\sqrt{1/(\ln L)}\right)^t\right).~~~~~~~~~~~
\end{split}
\end{equation}   
\end{theorem}
\begin{proof}  
We start the proof by estimating $\left\|\frac{\partial^{k} (Tf)}{\partial t_{j}^{k} }\right\|_{L^{1}([0,1]^{d})}$.  Let $(g\circ h)(y):=\exp\left(\frac{-1}{y(1-y)}\right)$,  where $g(x)=e^x$, $h(x)=\frac{-1}{x(1-x)}$. 
For every $k\in \mathbb{N}$, using Lemma \ref{bucc}, we have
\begin{equation}\label{leisil}
\begin{split}
 \sup_{y\in (0,1)}\left|\frac{d^{k} \psi}{d y^{k} }(y) \right|= \sup_{y\in (0,1/2]}\left|\frac{d^{k} \psi}{d y^{k} }(y) \right|&\le   \sup_{y\in (0,1/2]}\left|\frac{4^{k}k!\left(\frac{1}{y^{2k}}+\frac{1}{(1-y)^{2k}}\right)\exp\left(-\frac{1}{y}-\frac{1}{1-y}\right) }{\int_{0}^{1} \exp(-1/(\xi(1-\xi)))d\xi}  \right| 
 \\ &\le \sup_{y\in (0,1/2]}\frac{c_{1} 4^{k}k!  }{y^{2k}}\exp\left(\frac{-1}{y}\right)=c_{1}( 4 k/e)^{ 2k } k! .
 \end{split}
\end{equation}
Here, $c_{1}= 2 e^2/\int_{0}^{1} \exp(-1/(\xi(1-\xi)))d\xi$ represents an absolute constant. 
For the last inequality in  (Eq.~(\ref{leisil})), we have used the fact that the function  $x^{2k}/e^x$  attains its maximum on the interval \((1,+\infty)\) at \(x =  2k \).

Combining (Eq.~(\ref{ta1e1})), (Eq.~(\ref{leisil})) with  Lemma \ref{bucc}, we deduce that
\begin{equation}\label{leisi2}
 \left\|\frac{\partial^{k} f_{\psi}}{\partial t_{j}^{k} }\right\|_{L^{1}([0,1]^d)} \le  4^{k} C_{f,1}(C_{f,2})^{k}k! (c_{1}+1)^k( 4k )^{2k} k! \quad \text{for} \quad \text{every} \quad k\in \mathbb{N}.
\end{equation}
From (Eq.~(\ref{leisil})) and (Eq.~(\ref{leisi2})), by applying Leibniz's formula and Stirling's formula, we can prove that there exist absolute constants $C_1$ and $C_2$ such that 
\begin{equation}\label{ta1e2}
 \left\|\frac{\partial^{k} (Tf)}{\partial t_{j}^{k} } \right\|_{ {L^{1}([0,1]^d)}} \le C_{f,1}(C_{1}C_{f,2}k)^{C_{2}(k+d)} \quad \text{for}~\text{every} ~k\in \mathbb{N}~\text{and}~ 1\le j\le d.
\end{equation}
Employing  (Eq.~(\ref{ta1e2})) and integration by parts, we obtain
$$|\widehat{Tf}(\boldsymbol{w}) |\le C_{f,1}(C_{1}C_{f,2}k)^{C_{2}(k+d)}/ \|\boldsymbol{w}\|_{\infty}^k\quad \text{for} \quad  \text{every} \quad k\in \mathbb{N} ,$$
where $\boldsymbol{w}=(w_{1},w_{2},\dots,w_{d})\in \mathbb{N}_{0}^{d}\backslash\{\boldsymbol{0}\}$ and $ \|\boldsymbol{w}\|_{\infty}:=\max_{1\le j\le d}\{|w_{1}|,|w_{2}|,\dots,|w_{d}|\}$. 

By choosing $k=\max\left\{\left\lfloor\frac{( \|\boldsymbol{w}\|_{\infty}/e)^{1/(2C_{2})}}{C_{1}C_{f,2}}\right\rfloor,d\right\}$,
when $\frac{( \|\boldsymbol{w}\|_{\infty}/e)^{1/(2C_{2})}}{C_{1}C_{f,2}}\ge d+1$, we derive
$$|\widehat{Tf}(\boldsymbol{w}) |\le C_{f,1}\exp\left(-\left\lfloor\frac{( \|\boldsymbol{w}\|_{\infty}/e)^{1/(2C_{2})}}{C_{1}C_{f,2}}\right\rfloor\right).$$ 
Consequently, there exist positive constants $C_3$ and $C_{f,3}$ such that 
 $$|\widehat{Tf}(\boldsymbol{w}) |\le C_{f,1}e^{-\frac{  \|\boldsymbol{w}\|_{\infty}^{C_{3}}}{C_{f,3}}} \quad \text{for} \quad \frac{  \|\boldsymbol{w}\|_{\infty}^{C_{3}}}{C_{f,3}} \ge d+1. $$ 
Next, let $\beta>1$, $\rho>1$, and $\lambda>0$. Through a series of deductions, we obtain $$\sum_{m=1}^{\infty}m^{d}\rho^{-(m\beta)^{\lambda}}=\rho^{-(\beta)^{\lambda}}\sum_{m=1}^{\infty}m^{d}\rho^{-((m\beta)^{\lambda}-\beta^{\lambda})}\le \rho^{-(\beta)^{\lambda}}\sum_{m=1}^{\infty}m^{d}\rho^{- (m^{\lambda}-1)}=C_{\rho,\lambda,d}\rho^{-(\beta)^{\lambda}},$$
where $C_{\rho,\lambda,d}:=\sum_{m=1}^{\infty}m^{d}\rho^{- m^{\lambda}+1}$. Therefore, for $\beta>1$, there exists a constant $C_{f,4,d}$ depending on $f$ and $d$, and a constant $C_{f,5}$ depending on $f$, such that
\begin{equation*}
\begin{split}
\bigg(\sum_{ \|\boldsymbol{w}\|_{\infty}> \beta}|\widehat{Tf}(\boldsymbol{w}) |^{2}\bigg)^{1/2}&\le \bigg(\sum_{j=1}^{ +\infty }\sum_{(j+1)\beta\ge  \|\boldsymbol{w}\|_{\infty}>j \beta}|\widehat{Tf}(\boldsymbol{w}) |^{2}\bigg)^{1/2}\le \bigg(\sum_{j=1}^{ +\infty }(j+1)^{d}\beta^{d}C_{f,1}e^{-\frac{ (j\beta)^{C_{3}}}{C_{f,3}}}\bigg)^{1/2}\\
&\le C_{f,4,d}(\beta)^{d/2}e^{-\frac{ \beta^{C_{3}}}{C_{f,5}}}.
\end{split} 
\end{equation*}
Let $K\ge2$ be a positive integer, and define $\beta_{K,d}=\left((K-1)^{1/d}-1\right)/2$. We decompose $Tf$ as 
$$Tf(\boldsymbol{x})=a_{\boldsymbol{0}}+g(\boldsymbol{x})+R(\boldsymbol{x}),$$
where $$g(\boldsymbol{x})=\sum_{\boldsymbol{w}\in  \mathbb{Z}^{d}\cap[-\beta_{K,d},\beta_{K,d}]^{d} \backslash\{\boldsymbol{0}\}  }a_{\boldsymbol{w}}\exp(2\pi i \boldsymbol{w}\cdot \boldsymbol{x}),$$
, $$R(\boldsymbol{x})=\sum_{\boldsymbol{w} \notin  \mathbb{Z}^{d}\cap[-\beta_{K,d},\beta_{K,d}]^{d} }a_{\boldsymbol{w}}\exp(2\pi i \boldsymbol{w}\cdot \boldsymbol{x}).$$
Then
$$\#\{\boldsymbol{w}| \boldsymbol{w}\in  \mathbb{Z}^{d}\cap[-\beta_{K,d},\beta_{K,d}]^{d} \backslash\{\boldsymbol{0}\}  \}\le K,$$ and $$\|R\|_{L^{2}([0,1]^d)}^2 \le C_{f,4,d}(\beta_{K,d})^{d/2}e^{-\frac{ \beta_{K,d}^{C_{3}}}{C_{f,5}}}\le C_{f,4,d} \sqrt{K}e^{-\frac{ (K-1)^{C_{3}/d}}{4^{C_{3}}C_{f,5}}}.$$
By choosing   $ \sigma =\sigma_{2}=1/(2\ln L)$ in Theorem \ref{T1}, we successfully complete the proof.
\end{proof}
\begin{remark}\label{jiexi}
By substituting  \(\epsilon = \exp\left(-\frac{L^{C_{*}/d}}{C_{f,*}}\right)\) into (Eq.~(\ref{T991})), we demonstrate that the error exhibits semi-exponential behavior in relation to \(L\), a property that holds with a probability of at least \(1 - \frac{1}{\ln L}\). This result is comparable to Theorem 2 presented in \cite{pan2024super}. Furthermore, the results obtained from Eq.~(\ref{T993}) can be compared with Corollary 3 in \cite{pan2024super}. In \cite{pan2024super}, the upper bound for the convergence rate is articulated as \(C_{f}M^{-C\log(M)/d}\), where \(C_{f}\) is a constant dependent on \(f\) and \(d\), and \(C\) is an absolute constant such that \(C < 0.21\). Asymptotically, our convergence order significantly surpasses the findings reported in \cite{pan2024super}, despite the fact that the constant in the exponential term of our error bounds is contingent upon \(f\).  Moreover, our algorithm also achieves sub-exponential convergence for the  infinitely differentiable functions   as defined in \cite{suzuki2017super,dick2017construction}. 
\end{remark}
\begin{remark}
From the standpoint of sample complexity, our theorem establishes the existence of $M$ samples that can achieve a semi-exponential convergence rate. However, to achieve high precision with this method, it also involves approximating the integral
$\int_{0}^{y} \exp(-1/( \xi (1- \xi )))d \xi$. For future investigations, we propose an alternative framework: for any specified sample size $M$, we will employ B-splines to construct a $k_{M}$-order smooth function $\psi_{k_{M}}$ as a substitute for the function $\psi$ defined in (Eq.~(\ref{ta1e99})).
\end{remark}

We further deduce the following theorem concerning the subspace of the Wiener algebra introduced in \cite{goda}: 
\begin{equation*}\mathcal{A_{\kappa}}:= \left\{f\in C(\mathcal{T}^d)\middle| \|f\|_{\mathcal{A}_{\kappa}}:=  |\hat{f}(\boldsymbol{0})|+ \sum_{\boldsymbol{k}=(k_1,k_2,\dots,k_{d})\in \mathbb{Z}^{d} \backslash\{\boldsymbol{0} \}} |\hat{f}(\boldsymbol{k})\kappa(\min_{j\in\supp(\boldsymbol{k})}|k_{j}|)  |< \infty\right\},\end{equation*}
where  $\kappa(x)\ge 1$  is an increasing function on $(0,+\infty)$. For any $0<\epsilon \le 1/2$, there exists a positive number $N_{\kappa,\epsilon}$ such that $\kappa(N_{\kappa,\epsilon})\ge 1/\epsilon$.

Note that in the following theorem, we use $I(f)_{\boldsymbol{H}_{N_j},L,r,\boldsymbol{z}_{j}}$ instead of $I(f_{N,\boldsymbol{\eta}_{j}})_{\boldsymbol{H}_{N_j},L,r,\boldsymbol{z}_{j}}$.
\begin{theorem}\label{weina}
Let $f\in \mathcal{A_{\kappa}}$ with $\|f\|_{\mathcal{A}_{\kappa}}\le1$ . Let $n>2$, $0<\sigma \le 1/10$, $0<\epsilon <\sigma/n$. Define $B=\lceil4n/\sigma\rceil$, $r=B \sqrt{\log (1 / \epsilon)}$, $L=\lceil r\sqrt{2\log (1/\epsilon)}\rceil$. Let  $N\ge \max\{\lceil N_{\kappa,\epsilon} \rceil, 3L\}$ be a prime number. Then, over the randomness of $\boldsymbol{H}_{N}$ and $\boldsymbol{z}$, with
probability at least $1-\sigma$, it holds
\begin{equation}\label{T41}\bigg|I(f)_{\boldsymbol{H}_{N},L,r,\boldsymbol{z}}- \operatorname{INT}(f) \bigg |\le \Theta\left(\frac{\log(1/\epsilon)}{L\sigma} +\epsilon\right).\end{equation} 
For the ${\rm median}\{I(f)_{\boldsymbol{H}_{N_j},L,r,\boldsymbol{z}_{j}}\}_{j=1}^{t}$ (as defined in Theorem \ref{T1}), with probability at least $1-(2\sqrt{\sigma})^t$, 
\begin{equation}\label{T42}\big|{\rm median}\{I(f)_{\boldsymbol{H}_{N_j},L,r,\boldsymbol{z}_{j}}\}_{j=1}^{t}-\operatorname{INT}(f)\big| \le \Theta\left(\frac{\log(1/\epsilon)}{L\delta}+\epsilon\right). \end{equation}
Furthermore, 
\begin{equation}\label{T43}
\begin{split}
\bigg(\mathbb{E}_{\{\boldsymbol{H}_{N_j},\boldsymbol{z}_{j},\boldsymbol{\eta}_{j}\}_{j=1}^t}\big|{\rm median}\{I(f)_{\boldsymbol{H}_{N_j},L,r,\boldsymbol{z}_{j}}\}_{j=1}^{t}-\operatorname{INT}(f)\big|^{2}\bigg)^{1/2} \le \Theta\left(\frac{\log(1/\epsilon)}{L\sigma}+ \epsilon+ \left(2\sqrt{\sigma}\right)^t\right).
\end{split}
\end{equation}    
\end{theorem}
\begin{proof}
Let $f(\boldsymbol{x}):=\sum_{\boldsymbol{w}\in\mathbb{Z}^d} a_{\boldsymbol{w}}\exp(2\pi i \boldsymbol{w}\cdot \boldsymbol{x})$, then $\sum_{\boldsymbol{w}\in\mathbb{Z}^d} |a_{\boldsymbol{w}}|\le 1$. Let $$\boldsymbol{\xi}_{N}=\{(t_{1}N,t_{2}N,\dots,t_{d}N)|(t_{1} ,t_{2},\dots,t_{d})\in \mathbb{Z}^d\backslash\{\boldsymbol{0}\}\}.$$
Since $N\ge  \lceil N_{\kappa,\epsilon} \rceil $, we have  $$|\widetilde{f}^{(N)}(\boldsymbol{0})-a_{\boldsymbol{0}}|\le \sum_{\boldsymbol{w}\in \boldsymbol{\xi}_{N}}|a_{\boldsymbol{w}}|\le (1/\kappa(N_{\kappa,\epsilon}))\sum_{\boldsymbol{w}\in \boldsymbol{\xi}_{N}}|a_{\boldsymbol{w}}|\kappa(\min_{j\in\supp(\boldsymbol{w})}|w_{j}|)\le \epsilon,$$
where $\widetilde{f}^{(N)}(\boldsymbol{w})$ is defined in (Eq.~(\ref{dft})). Combining the above inequality
and (Eq.~(\ref{l23e11})), we have
$$\left|\sum_{|l|\le L} \widetilde{f}^{(N)}(\boldsymbol{0}) G_{r,l}-a_{\boldsymbol{0}}\right|  \le \Theta(\epsilon ).$$
Decompose $f$ as \begin{equation}\label{WW}
f(\boldsymbol{z}/N)=\widetilde{f}^{(N)}(\boldsymbol{0})+g(\boldsymbol{z}/N).\end{equation}
Evidently, $\widetilde{g}^{(N)}(\boldsymbol{0})=0$. Let $b_{\boldsymbol{w}}=\widetilde{g}^{(N)}(\boldsymbol{w})$, then  $\sum_{\boldsymbol{w}\in \mathbb{Z}_{N}^{d}\backslash\{\boldsymbol{0}\}}|b_{\boldsymbol{w}}|\le1$ . 
Combining (Eq.~(\ref{l23e11})) and (Eq.~(\ref{unb})), we obtain
$$ \mathbb{E}_{\boldsymbol{H}_{N}}\bigg(\Big|\sum_{|l|\le L} \exp(2\pi i \boldsymbol{w}\cdot (\boldsymbol{z}-l\boldsymbol{H}_{N})/N)G_{r,l}\Big|\bigg|\boldsymbol{w}\in U_{\boldsymbol{H}_N,B} \bigg)  \mathbb{P}\Big\{\boldsymbol{w}\in U_{\boldsymbol{H}_N,B}\Big\}\le  \Theta (1/B).
$$
Using (Eq.~(\ref{l23e10})), we have 
$$ \mathbb{E}_{\boldsymbol{H}_{N}}\bigg(\Big|\sum_{|l|\le L} \exp(2\pi i \boldsymbol{w}\cdot (\boldsymbol{z}-l\boldsymbol{H}_{N})/N)G_{r,l}\Big|\bigg|\boldsymbol{w}\notin U_{\boldsymbol{H}_N,B} \bigg) \mathbb{P}\Big\{\boldsymbol{w}\notin U_{\boldsymbol{H}_N,B}\Big\}\le  \Theta (\epsilon(1-1/B)).
$$
Therefore,
\begin{equation*}
\begin{split}
&\mathbb{E}_{\boldsymbol{H}_{N}}\left|\sum_{|l|\le L}g\left(\left\{\frac{\boldsymbol{z}-l\boldsymbol{H}_{N}}{N}\right\}\right) G_{r,l}\right|\\
&\le \sum_{\boldsymbol{w}\in \mathbb{Z}_{N}^{d}\backslash\{\boldsymbol{0}\}}\mathbb{E}_{\boldsymbol{H}_{N}}\left|b_{\boldsymbol{w}}\right|\Big|\sum_{|l|\le L}\exp(2\pi i \boldsymbol{w}\cdot (\boldsymbol{z}-l\boldsymbol{H}_{N})/N)G_{r,l}\Big|\\
&=\sum_{\boldsymbol{w}\in \mathbb{Z}_{N}^{d}\backslash\{\boldsymbol{0}\}}\bigg(\left|b_{\boldsymbol{w}}\right| \mathbb{E}_{\boldsymbol{H}_{N}}\bigg(\Big|\sum_{|l|\le L} \exp(2\pi i \boldsymbol{w}\cdot (\boldsymbol{z}-l\boldsymbol{H}_{N})/N)G_{r,l}\Big|\bigg|\boldsymbol{w}\in U_{\boldsymbol{H}_N,B} \bigg)  \mathbb{P}\Big\{\boldsymbol{w}\in U_{\boldsymbol{H}_N,B}\Big\}\\
&+|b_{\boldsymbol{w}} |\mathbb{E}_{\boldsymbol{H}_{N}}\left(\Big|\sum_{|l|\le L} \exp(2\pi i \boldsymbol{w}\cdot (\boldsymbol{z}-l\boldsymbol{H}_{N})/N)G_{r,l}\Big|\bigg|\boldsymbol{w}\notin U_{\boldsymbol{H}_N,B} \right) \mathbb{P}\Big\{\boldsymbol{w}\notin U_{\boldsymbol{H}_N,B}\Big\}\bigg)\\&\le \Theta (1/B+\epsilon)\sum_{\boldsymbol{w}\in \mathbb{Z}_{N}^{d}\backslash\{\boldsymbol{0}\}}|b_{\boldsymbol{w}}|\le \Theta (1/B+\epsilon) .
\end{split}
\end{equation*}
Recalling the conditions $\epsilon<\sigma/n$ and $B=\lceil4n/\sigma\rceil$, we apply Markov's inequality to obtain:
\begin{equation}\label{W41}
\mathbb{P}\left\{ \left|\sum_{|l|\le L}g\bigg(\bigg\{\frac{\boldsymbol{z}-l\boldsymbol{H}_{N}}{N}\bigg\}\bigg) G_{r,l}\right|  \le \Theta\bigg(\frac{1}{n} \bigg)\right\}\ge 1-\sigma.
\end{equation}
Therefore, (Eq.~(\ref{T41})) and  (Eq.~\eqref{T42}) can be derived by combining (Eq.~(\ref{WW})) and  (Eq.~\eqref{W41}).  By employing a method analogous to the proof of (Eq.~\eqref{253}), we establish (Eq.~(\ref{T43})).
\end{proof}
\begin{remark}\label{recon}
If we further assume that $f$ is H\"{o}lder continuous (see \cite{chen2024information,Dickwl}), a similar conclusion remains valid. In this scenario, it suffices to replace $N_{\kappa,\epsilon}$ with a quantity dependent on the H\"{o}lder continuity parameters and $\epsilon$.
\end{remark}
\begin{remark}\label{reweina}
By (Eq.~(\ref{1021})) and (Eq.~(\ref{T43})), the sample size $M=t(2L+1)$. By setting $\sigma=1/10$, $t=\log_{\sqrt{5}/2} n$, and $\epsilon=1/n$ in (Eq.~(\ref{T43})), we derive that the upper bound of the RMSE error is $\Theta((\log M)^2/M)$.
\end{remark}
\begin{remark}
  By employing the results on the best n-term approximation in the $L^2$ sense (see, e.g., Lemma 1 in \cite{barron1993universal}), the functions in $\mathcal{A_{\kappa}}$ can be reformulated as the form \eqref{fcsp}. This reformulation then permits the application of Theorem \ref{T1} to obtain the conclusion of Theorem \ref{weina} by using $I(f_{N,\boldsymbol{\eta}_{j}})_{\boldsymbol{H}_{N_j},L,r,\boldsymbol{z}_{j}}$ instead of $I(f)_{\boldsymbol{H}_{N_j},L,r,\boldsymbol{z}_{j}}$.
\end{remark}

\section{Integration in Weighted Sobolev Spaces}
In this section, we provide a theoretical analysis of our integration algorithm within the weighted Sobolev space defined as
 $$H_{\gamma,s}(\mathcal{T}^d)=\left\{ f \in L^{2}(\mathcal{T}^d) \middle|\|f\|_{\gamma,s}^{2}:=\sum_{\boldsymbol{w}=(w_{1},\dots,w_{d})\in \mathbb{Z}^d}(\Gamma_{\gamma,s}(\boldsymbol{w}))^{2}|\hat{f}(\boldsymbol{w})|^2< \infty \right\},$$
where 
$$\Gamma_{\gamma,s}(\boldsymbol{w})=(\gamma_{\supp(\boldsymbol{w})})^{s}\prod_{j=1}^{d}\max \left\{1,  \left|w_j\right|^s \right\},$$
and  $\gamma_{u}:u\subseteq\{1,2,\dots,d\}\rightarrow [1,+\infty)$ is a weight function. Note that $\gamma_{\supp(\boldsymbol{0})}=\gamma_{\emptyset}$, $\emptyset\subseteq \{1,2,\dots,d\}$.
\begin{theorem}\label{T3} Let $f\in H_{\gamma,s}(\mathcal{T}^d)$ with  $\|f\|_{\gamma,s}\le 1$ , $s> 0$, and $\alpha \ge 1$. Let $0<\sigma\le 1/10$, $n>1$, $0<\epsilon <\sigma/n$. Let $B=\left\lceil 4n/\sigma\right\rceil$, $r=B \sqrt{\log \left(1/\epsilon\right)}$, $L=\left\lceil r\sqrt{2\log \left(1/\epsilon\right)}\right\rceil$, and  $N=\max\{\Theta(1/(\sqrt{\sigma}\epsilon)^{2/d}),\Theta((L^{2/d}B^2)/\epsilon^{4/d}),3L\}$ be a prime number. Under the conditions $L< N/3$, with probability at least  $1-\sigma-1/10$ over the randomness of $\boldsymbol{\eta}$, $\boldsymbol{H}_{N}$, and $\boldsymbol{z}$, it holds
\begin{equation}\label{T31}
\begin{split}
\left|I(f_{N,\boldsymbol{\eta}})_{\boldsymbol{H}_{N},L,r,\boldsymbol{z}}- \operatorname{INT}(f) \right|\le \Theta\left(\frac{C_{\alpha,s,\gamma,N}(\log(1/\epsilon))^{s/\alpha+1 / 2}}{ (L\delta)^{s/\alpha+1 / 2}}+\epsilon\right).
\end{split}
\end{equation} 
For the ${\rm median}\{I(f)_{\boldsymbol{H}_{N_j},L,r,\boldsymbol{z}_{j}}\}_{j=1}^{t}$ (as defined in Theorem \ref{T1}), the following inequality holds with a probability of at least $1-(2\sqrt{\sigma+1/10})^t$:
\begin{equation}\label{T32}
\begin{split}
\left|{\rm median}\{I(f_{N,\boldsymbol{\eta}_{j}})_{\boldsymbol{H}_{N_j},L,r,\boldsymbol{z}_{j}}\}_{j=1}^{t}-\operatorname{INT}(f)\right| \le \Theta\left(\frac{C_{\alpha,s,\gamma,N}(\log(1/\epsilon))^{s/\alpha+1 / 2}}{ (L\delta)^{s/\alpha+1 / 2}}+\epsilon\right).
\end{split}
\end{equation}
Furthermore, it holds
\begin{equation}\label{T33}
\begin{split}
&\left(\mathbb{E}_{\{\boldsymbol{H}_{N_{j}},\boldsymbol{z}_{j},\boldsymbol{\eta}_{j}\}_{j=1}^t}\big|{\rm median}\{I(f_{N,\boldsymbol{\eta}})_{\boldsymbol{H}_{N_j},L,r,\boldsymbol{z}_{j}}\}_{j=1}^{t}-\operatorname{INT}(f)\big|^{2}\right)^{1/2} \\&\le \Theta\left(\frac{C_{\alpha,s,\gamma,N}(\log(1/\epsilon))^{s/\alpha+1 / 2}}{ (L\delta)^{s/\alpha+1 / 2}}+\epsilon+   (2\sqrt{\sigma+1/10})^{t}\right),~~~~~~~~~~~
\end{split}
\end{equation}
where $
C_{\alpha,s,\gamma,N}:= \left(\sum_{\boldsymbol{w}\in  Q_{N/2}\backslash\{\boldsymbol{0}\}}\frac{1}{(\Gamma_{\gamma,s}(\boldsymbol{w}))^{\alpha/s}}\right)^{s/\alpha}+1$ and $Q_{N/2}:=\mathbb{Z}_{N}^d\cap(-N/2,N/2)^d$.
\end{theorem}
\begin{remark}\label{piu} By (Eq.~(\ref{1021})) and (Eq.~(\ref{T33})), the sample size is $M=t(2L+1)$. By setting $\sigma=1/10$, $t=\log_{\sqrt{5}/2} (n^{s+1/2})$, and $\epsilon=1/n^{s+1/2}$ in (Eq.~(\ref{T33})), we derive that the upper bound of the RMSE error is: 
\begin{equation}\label{sp3333}
C'C_{\alpha,s,\gamma,N} (3+4s)^{4s+3}(\log M)^{2s+1} /M^{s+1/2},
\end{equation}
where $C'$ denotes an absolute constant.  
\begin{itemize}
    \item When $\alpha=1$ and 
    $\gamma_{\supp(\boldsymbol{w})}=1$, $C_{\alpha,s,\gamma,N}\le (\Theta(1)\log M) )^d$.
    \item  When $\alpha>1$ and $\gamma_{\supp(\boldsymbol{w})}=\prod_{j\in \supp(\boldsymbol{w})}(1/\gamma_{j})$ with $\gamma_{j}\le 1$, the RMSE error \eqref{sp3333} is comparable to that in \cite{kritzer2019lattice}, showing no clear superiority between the two algorithms. Both algorithms are polynomial tractable, provided that $\sum _{j=1}^{d}\gamma_{j}:=C_{\gamma}$ and $\zeta(\alpha)$ are treated as constants (as in \cite{kritzer2019lattice}), where $\zeta$ denotes the Riemann zeta function.
\end{itemize}
Compared to \cite{kritzer2019lattice}, our algorithm does not require prior information on the weights. Moreover, if the prior information of the order $s$ is  unknown, setting $\epsilon=1/(n^{\log n})$ and $t=\log_{\sqrt{5}/2} n^{\log n})$ renders our algorithm universal for $s$.
\end{remark}
\begin{remark}\label{goda2} Recently, \cite{goda2024simple} proposed a nearly optimal algorithm for the case where the order $s>1/2$ and $ \sum _{j=1}^{d}\gamma_{j}:=C_{\gamma}< \infty$. Their approach also avoids requiring prior knowledge of the weights or the order $s$. Comparable to \cite{goda2024simple}, our algorithm applies to the scenario $s>0$, making it suitable for discontinuous and non-periodic functions.
As noted earlier, these two works are independent of each other and different in terms of algorithms.
\end{remark}
\begin{proof}
Our proof draws inspiration from  \cite{kritzer2019lattice}. For $\alpha \ge 1$, define
$$V_{d,\alpha,N}:=\sum_{\boldsymbol{w}\in  Q_{N/2}\backslash\{\boldsymbol{0}\}}\frac{1}{(\Gamma_{\gamma,s}(w))^{\alpha/s}}.$$
By Lemma \ref{seeeaa}, for $N/3 \ge B\ge 4n/\sigma$,
$$
\mathbb{E}_{\boldsymbol{H}_{N}}\sum_{\boldsymbol{w}\in  Q_{N/2}\backslash\{\boldsymbol{0}\}}\mathbb{I}_{\boldsymbol{w}\in U_{\boldsymbol{H}_{N},B}}\frac{1}{(\Gamma_{\gamma,s}(\boldsymbol{w}))^{\alpha/s}}\le  \frac{\sigma V_{d,\alpha,N}}{n}.
$$
By Markov’s inequality, the inequality
$$\sum_{\boldsymbol{w}\in  Q_{N/2}\backslash\{\boldsymbol{0}\}}\mathbb{I}_{\boldsymbol{w}\in U_{\boldsymbol{H}_{N},B}}\frac{1}{(\Gamma_{\gamma,s}(\boldsymbol{w}))^{\alpha/s}}\le  \frac{\sigma V_{d,\alpha,N}}{n\sigma}=  \frac{V_{d,\alpha,N}}{n}$$
holds with probability at least $1-\sigma$ over $\boldsymbol{H}_{N}$. Consequently, with probability at least $1-\sigma$, 
\begin{equation}\label{prt3}
\boldsymbol{w}\notin U_{\boldsymbol{H}_{N},B} \quad \text{for} \quad \text{all} \quad \boldsymbol{w}\in\left\{\boldsymbol{\xi}\in  Q_{N/2}\backslash\{\boldsymbol{0}\} \middle|(\Gamma_{\gamma,s}(\boldsymbol{\xi}))^{\frac{1}{s}}<\bigg(\frac{n}{V_{d,\alpha,N}}\bigg)^{\frac{1}{\alpha}}\right\}.
\end{equation}
Let 
\begin{equation}\label{bcdy}  
U_{1}:=\left\{\boldsymbol{w}\in \mathbb{Z}^{d}\cap [-B,B]^d\backslash\{\boldsymbol{0}\}\middle|(\Gamma_{\gamma,s}(\boldsymbol{w}))^{\frac{1}{s}}<\left(\frac{n}{V_{d,\alpha,N}}\right)^{\frac{1}{\alpha}}\right\}.
\end{equation}
 Recall that $N/3 \ge B\ge 4n/\sigma$, by (Eq.~\eqref{prt3}) and (Eq.~\eqref{bcdy}), with probability at least $1-\sigma$,
\begin{equation}\label{aijia}
\boldsymbol{w}\notin U_{\boldsymbol{H}_{N},B} \quad \text{for} \quad \boldsymbol{w}\in U_1.
\end{equation}

For any $f\in H_{\gamma,s}(\mathcal{T}^d)$, decompose $f$ as
$$f(\boldsymbol{x})=a_{\boldsymbol{0}}+g(\boldsymbol{x})+R(\boldsymbol{x}),$$
where
$$a_{\boldsymbol{0}}=\operatorname{INT}(f),g(\boldsymbol{x})=\sum_{\boldsymbol{w}\in  U_1}a_{\boldsymbol{w}}\exp(2\pi i \boldsymbol{w}\cdot \boldsymbol{x}),~\text{and}~R(\boldsymbol{x})=\sum_{\boldsymbol{w}\in \mathbb{Z}^{d} \backslash (U_{1}\cup\{\boldsymbol{0}\}) }b_{\boldsymbol{w}}\exp(2\pi i \boldsymbol{w}\cdot \boldsymbol{x}).$$ 
Since $ \gamma_{supp(\boldsymbol{w})}\in [1,\infty)$,
\begin{equation*}
\begin{split}
\|R\|_{L^{2}([0,1]^d)}\le \frac{1}{B^s}+\left(\frac{V_{d,\alpha,N}}{n}\right)^{s/\alpha}\le \frac{(V_{d,\alpha,N})^{s/\alpha}+1}{n^{s/\alpha}}\le\Theta\left(\frac{ C_{\alpha,s,\gamma,N}(\log(1/\epsilon))^{s/\alpha}}{(L\sigma)^{s/\alpha}}\right).
\end{split}
\end{equation*}
 By Theorem \ref{T1}, 
\begin{equation}\label{fff8}
\begin{split}
\mathbb{P}\left\{\left|I((a_{\boldsymbol{0}}+R)_{N,\boldsymbol{\eta},\boldsymbol{H}_{N},L,r,\boldsymbol{z}}-a_{\boldsymbol{0}} \right|\le  \Theta\left( \epsilon+ \frac{C_{\alpha,s,\gamma,N}(\log(1/\epsilon))^{s/\alpha+1/2}}{(L\sigma)^{s/\alpha+1/2}}
\right) \right\}\ge 1-\sigma/2-1/30.
\end{split}
\end{equation} 
  By (Eq.~\eqref{tdingyi}), (Eq.~\eqref{aijia}), and (Eq.~\eqref{whm}), we have with probability at least $1-\sigma$ that
\begin{equation}\label{aijia2}
\mathbb{E}_{\boldsymbol{H}_{N}}\big(t_{\boldsymbol{w}}|\boldsymbol{w}\in U_{1}\big)\le \Theta(\epsilon^2 ).
\end{equation}
According to (Eq.~\eqref{aijia2}) and (Eq.~\eqref{zj}), we have with probability at least $1-\sigma$ that
\begin{equation*}
\begin{split}
\mathbb{E}_{\boldsymbol{H}_{N}}\mathbb{E}_{\boldsymbol{z}}\left|\sum_{|l|\le L}g\bigg(\bigg\{\frac{\boldsymbol{z}-l\boldsymbol{H}_{N}}{N}\bigg\}\bigg) G_{r,l}\right|^2=\mathbb{E}_{\boldsymbol{H}_{N}} \sum_{\boldsymbol{w}\in  U_{1} }\left|a_{\boldsymbol{w}}\right|^{2}t_{\boldsymbol{w}}\le \Theta(\epsilon^2) .
\end{split}
\end{equation*} 
Applying Markov’s inequality gives
\begin{equation}\label{21}
    \bigg|\sum_{|l|\le L}g\bigg(\bigg\{\frac{\boldsymbol{z}-l\boldsymbol{H}_{N}}{N}\bigg\}\bigg) G_{r,l}\bigg|^2\le 30\epsilon^2 ,
\end{equation}
with probability at least $1-\sigma-1/30$ over the randomness of  $\boldsymbol{H}_{N}$ and $\boldsymbol{z}$. 

Using the first inequality in  (Eq.~\eqref{lem250102}),  it holds that
$$ \left|(g-g_{N,\boldsymbol{\eta}})\left(\left\{\frac{\boldsymbol{z}-l\boldsymbol{H}_{N}}{N}\right\}\right) \right|^{2}\le \frac{6(2B+1)^d}{\sigma_{2}N^{d/2}}\quad \text{for}\quad \text{every}\quad l\in \mathbb{Z}\cap [-L,L],$$
 with probability at least $1-(2L+1)\cdot\sigma_2/2$ over the randomness of  $\boldsymbol{z}$ and $\boldsymbol{\eta}$. 

Recalling the condition $N\ge \Theta((L^{2/d}B^2)/\epsilon^{4/d})$ and choosing $\sigma_{2}=\frac{1}{30(2L+1)}$, we have
$$ \left|\sum_{|l|\le L}(g-g_{N,\boldsymbol{\eta}})\left(\left\{\frac{\boldsymbol{z}-l\boldsymbol{H}_{N}}{N}\right\}\right) G_{r,l}\right|\le \Theta(\epsilon),$$
with probability at least $1-\frac{1}{30}$ over the randomness of  $\boldsymbol{z}$ and $\boldsymbol{\eta}$.
Combining the above inequality and (Eq.~\eqref{21}), we obtain 
\begin{equation}\label{fff9}
\mathbb{P}\left\{\left|I(g)_{N,\boldsymbol{\eta},\boldsymbol{H}_{N},L,r,\boldsymbol{z}} \right|\le  \Theta\left( \epsilon\right) \right\}\ge 1-\sigma-2/30.
\end{equation} 
 Combining (Eq.~\eqref{fff8}) and (Eq.~\eqref{fff9}), we see that the total failure probability is at most $2\sigma+1/10$. Replacing $\sigma$ by  $\sigma/2$, we complete the proof of   (Eq.~\eqref{T31}) and (Eq.~\eqref{T32}). By the same approach as in the proof of (Eq.~\eqref{253}), (Eq.~\eqref{T33}) holds true.
\end{proof}
\section{Numerical Experiments}
In this section, we employ the numerical examples from existing literature to demonstrate that
our algorithm is implementable and capable of achieving a convergence order comparable to those
reported in the literature. It is crucial to point out that while the objective of this paper is to
propose an algorithm that theoretically enjoys high convergence order, this does not imply that
our algorithm outperforms existing algorithms in terms of specific numerical examples. Further
analysis on this aspect is conducted in the ``Future work'' section.

For our numerical experiments, we will consider four test functions: a differentiable function $f_1$, a non-differentiable function $f_2$, a discontinuous indicator function $f_3$, and a high-frequency oscillatory function $f_4$. These functions are explicitly defined as follows:
\begin{equation*}
\begin{split}
f_1(\boldsymbol{x})&=\prod_{j=1}^d\left[1+\frac{1}{j^4}B_{4}(x_{j})\right],\\
f_{2}(\boldsymbol{x})&=\prod_{j=1}^d\left[1+\frac{\left|4 x_j-2\right|-1}{j^{ 4 }}\right],\\
f_{3}(\boldsymbol{x})&=\mathbb{I}_{\left\{\sum_{i=1}^d x_i \geq d / 2\right\}}(\boldsymbol{x}),\\
f_{4}(\boldsymbol{x})&=\prod_{j=1}^d\left[1+\frac{\left|4 x_j-2\right|-1}{j^4}\right]+\sin(20000\pi x_{1}),
\end{split}
\end{equation*}
where $B_{4}(y)=y^4-2y^3+y^2-\frac{1}{30}$ denotes the 4th-degree Bernoulli polynomial.

The examples $f_1$ and $f_2$, provided by  \cite{dick2022component} and \cite{goda2024simple} respectively, are employed to validate the nearly optimal performance of random algorithm for integration in weighted Sobolev spaces. Meanwhile, the example $f_3$ was utilized in \cite{he2016extensible,he2015convergence} to discuss the efficacy of integration algorithms
for discontinuous functions.  Although $f_3$ can still be considered as a function in weighted Sobolev
space, it falls outside the scope of the algorithm proposed in \cite{dick2022component,goda2024simple} due to its order being less than
$1/2$. Finally, the function $f_{4}$ is constructed by introducing a high-frequency term to $f_2$. Across all these examples, we set $d=20$. The squared error of our algorithm is measured by:
 $${\rm error}:=\big|{\rm median}\{I(f_{N,\boldsymbol{\eta}_{j}})_{\boldsymbol{H}_{N_j},L,r,\boldsymbol{z}_{j}}\}_{j=1}^{t}-\operatorname{INT}(f)\big|^2.$$
Here, the number of samples is denoted by $M:=2L+1$ (see (Eq.~(\ref{1021}))). Consistent with \cite{goda2022construction,goda2024simple,pan2025automatic}, we exclude the number of repetitions $t$ from the calculation of the total sample size. We next specify the selection of parameters for our algorithm. For the functions $f_{1}$, $f_{2}$, $f_3$, and $f_4$, we set  $N=5600748293801$.  This choice is motivated by the fact that increasing $N$ enables the algorithm to achieve higher target accuracy and enhance its capability to handle high-frequency components (see Theorems \ref{T1}, \ref{weina}, and \ref{T3}). Thus, it is advisable to select a sufficiently large prime \(N\), provided that machine precision, rounding errors, and analogous numerical artifacts can be neglected.  
Regarding the selection of $r$, we consider two cases: if prior information about the order $s$ is available, we set  $r=L/\sqrt{2(s+1/2)\log(2L+1)}$. Otherwise, we choose  $r=L/\sqrt{2\log\left((2L+1)\log(2L+1)\right)}$ (see Remark \ref{piu}). Finally, we uniformly select $t=2\lceil\log_{2}(2L) \log_{2}(\log_{2}(2L))/2\rceil+1$ for the repetitions.

\begin{table}[h]
	\small{\caption{\emph{Example $f_1$~-- 
Numerical results }}}
	\label{Tab1}	\centering
	\begin{tabular}{llllll|lll}
		\hline
		\multicolumn{5}{c}{Our method }& &\multicolumn{3}{c}{The Owen-scrambled Sobol' method} \\
		\hline
		 ~~~$M$   &~$t$  &~~~~$Mt$  &~~MSE &~Order&~  &~~~~$M$ &~~~~~~\quad\quad\quad~MSE &~~~~\quad\quad~Order  \\
		\hline
		~~~~5 &~3  &~~~~~15 &~1.76E-4 & && ~~~~~4 &~~~~~~\quad\quad2.31E-5  &\\ \hline
		~~~~9 &~7  &~~~~~63 &~1.13E-5 &~~2.02 && ~~~~~8 &~~~~~~\quad\quad3.26E-6  &~~~~~~~\quad\quad2.82\\ \hline
        ~~~17 &~9  &~~~~153 &~1.54E-6 &~~2.38&& ~~~~16&~~~~~~\quad\quad3.88E-7  &~~~~~~~\quad\quad3.07\\ \hline
        ~~~33 &~13  &~~~~429 &~1.69E-8 &~~4.49&& ~~~~32 &~~~~~~\quad\quad5.76E-8  &~~~~~~~\quad\quad2.75\\ \hline
        ~~~65 &17  &~~~1105 &2.81E-10 &~~4.40&& ~~~~64 &~~~~~~\quad\quad6.75E-9  &~~~~~~~\quad\quad3.09\\ \hline
	~~129 &21  &~~~2709 &1.15E-11&~~3.59 &&~~~128 &~~~~~\quad\quad9.13E-10  &~~~~~~~\quad\quad2.89\\ \hline
	~~257 &25  &~~~6425 &1.79E-13&~~4.85&&~~~256 &~~~~~\quad\quad9.97E-11  &~~~~~~~\quad\quad3.20\\ \hline
	~~513 &31 &~~16416 &5.67E-15 &~~3.82&&~~~512 &~~~~~\quad\quad1.29E-11  &~~~~~~~\quad\quad2.95\\ \hline
	~1025 &35  &~~35875 &3.89E-16 &~~3.30&&~~1024 &~~~~~\quad\quad1.60E-12  &~~~~~~~\quad\quad3.01\\ \hline
    ~2049 &41  &~~84009 &2.93E-18 &~~5.75 &&~~2048 &~~~~~\quad\quad1.75E-13  &~~~~~~~\quad\quad3.19\\ \hline
	~4097 &45  &~184365 &6.99E-20 &~~4.75&&~~4096 &~~~~~\quad\quad2.13E-14  &~~~~~~~\quad\quad3.04\\ \hline
    ~8193 &51 &~417843 &1.22E-21  &~~4.94&&~~8192 &~~~~~\quad\quad2.72E-15  &~~~~~~~\quad\quad2.97\\ \hline
    16385 &55  &~901175 &3.17E-23  &~~4.75&&~16384 &~~~~~\quad\quad3.56E-16  &~~~~~~~\quad\quad2.94\\ \hline
    32769 &61  &1998909 &6.22E-25 &~~4.94&&~32768 &~~~~~\quad\quad4.15E-17  &~~~~~~~\quad\quad3.10\\ \hline
    65537 &65  &4259905 &9.55E-27 &~~5.52&&~65536 &~~~~~\quad\quad8.11E-18  &~~~~~~~\quad\quad2.34\\
    	\hline
	\end{tabular}
\end{table}

\begin{table}[h]
	\small{\caption{\emph{Example $f_2$~-- 
Numerical results }}}
	\label{Tab2}	\centering
	\begin{tabular}{llllll|lll}
		\hline
		\multicolumn{5}{c}{Our method }& &\multicolumn{3}{c}{The Owen-scrambled Sobol' method} \\
		\hline
		 ~~~$M$   &~$t$  &~~~~$Mt$  &~~MSE &~Order&~  &~~~~$M$ &~~~~~~~\quad\quad\quad MSE &~~~~~\quad\quad Order  \\
		\hline
		~~~~5 &~3  &~~~~~15 &5.79E-2 & && ~~~~~4 &~~~~~~\quad\quad2.19E-2  &\\ \hline
		~~~~9 &~7  &~~~~~63 &2.30E-3 &~~2.37&& ~~~~~8 &~~~~~~\quad\quad2.94E-3  &~~~~~~~\quad\quad2.90\\ \hline
        ~~~17 &~9  &~~~~153 &6.97E-4 &~~1.42&& ~~~~16 &~~~~~~\quad\quad3.22E-4  &~~~~~~~\quad\quad3.19\\ \hline
        ~~~33 &13  &~~~~429 &2.50E-5 &~~3.32&& ~~~~32 &~~~~~~\quad\quad4.94E-5  &~~~~~~~\quad\quad2.70\\ \hline
        ~~~65 &17  &~~~1105 &1.66E-6 &~~2.91&& ~~~~64 &~~~~~~\quad\quad5.16E-6  &~~~~~~~\quad\quad3.26\\ \hline
	~~129 &21  &~~~2709 &1.95E-8&~~5.00 &&~~~128 &~~~~~~\quad\quad7.28E-7  &~~~~~~~\quad\quad2.83\\ \hline
	~~257 &25  &~~~6425 &1.40E-8&~~0.39&&~~~256 &~~~~~~\quad\quad1.00E-7  &~~~~~~~\quad\quad2.86\\ \hline
	~~513 &31 &~~16416 &1.35E-9 &~~2.59&&~~~512 &~~~~~~\quad\quad1.33E-8  &~~~~~~~\quad\quad2.91\\ \hline
	~1025 &35  &~~35875 &1.26E-10 &~~2.92&&~~1024 &~~~~~~\quad\quad1.54E-9  &~~~~~~~\quad\quad3.11\\ \hline
    ~2049 &41  &~~84009 &9.95E-12 &~~2.99 &&~~2048 &~~~~~\quad\quad2.60E-10  &~~~~~~~\quad\quad2.57\\ \hline
	~4097 &45  &~184365 &6.95E-13 &~~3.39&&~~4096 &~~~~~\quad\quad2.36E-11  &~~~~~~~\quad\quad3.46\\ \hline
    ~8193 &51 &~417843 &1.09E-13  &~~2.26&&~~8192 &~~~~~\quad\quad3.70E-12  &~~~~~~~\quad\quad2.67\\ \hline
    16385 &55  &~901175 &7.45E-15  &~~3.49&&~16384 &~~~~~\quad\quad3.32E-13  &~~~~~~~\quad\quad3.47\\ \hline
    32769 &61  &1998909 &7.93E-16 &~~2.81&&~32768 &~~~~~\quad\quad4.69E-14  &~~~~~~~\quad\quad2.82\\ \hline
    65537 &65  &4259905 &6.28E-17 &~~3.35&&~65536 &~~~~~\quad\quad7.40E-15  &~~~~~~~\quad\quad2.66\\
    	\hline
	\end{tabular}
\end{table}
\begin{table}[h]
	\small{\caption{\emph{Example $f_3$~-- 
Numerical results }}}
	\label{Tab3}	\centering
	\begin{tabular}{llllll|lll}
		\hline
		\multicolumn{5}{c}{Our method }& &\multicolumn{3}{c}{The Owen-scrambled Sobol' method} \\
		\hline
		 ~~~$M$   &~$t$  &~~~~$Mt$  &~~MSE &~Order&~  &~~~~$M$ &~~\quad\quad\quad MSE &~\quad\quad Order  \\
		\hline
		~~~~5 &~3  &~~~~~15 &3.01E-2 & && ~~~~~4 &\quad\quad 2.43E-2  &\\ \hline
		~~~~9 &~7  &~~~~~63 &8.90E-3 &~~0.90&& ~~~~~8 &\quad\quad 8.44E-3  &~~~\quad\quad 1.53\\ \hline
        ~~~17 &~9  &~~~~153 &4.91E-3 &~~0.71&& ~~~~16 &\quad\quad 3.60E-3  &~~~\quad\quad 1.23\\ \hline
        ~~~33 &13  &~~~~429 &1.42E-3 &~~1.25&& ~~~~32 &\quad\quad 1.91E-3  &~~~\quad\quad 0.91\\ \hline
        ~~~65 &17  &~~~1105 &4.61E-4 &~~1.19&& ~~~~64 &\quad\quad 8.52E-4  &~~~\quad\quad 1.17\\ \hline
	~~129 &21  &~~~2709 &1.66E-4&~~1.15 &&~~~128 &\quad\quad 6.29E-4  &~~~\quad\quad 0.44\\ \hline
	~~257 &25  &~~~6425 &8.63E-5 &~~0.76&&~~~256 &\quad\quad 2.70E-4  &~~~\quad\quad 1.22\\ \hline
	~~513 &31 &~~16416 &2.89E-5 &~~1.21&&~~~512 &\quad\quad 1.47E-4  &~~~\quad\quad 0.88\\ \hline
	~1025 &35  &~~35875 &1.53E-5 &~~0.78&&~~1024 &\quad\quad 8.76E-5  &~~~\quad\quad 0.74\\ \hline
    ~2049 &41  &~~84009 &5.22E-6 &~~1.26 &&~~2048 &\quad\quad 3.05E-5  &~~~\quad\quad 1.52\\ \hline
	~4097 &45  &~184365 &2.06E-6 &~~1.18&&~~4096 &\quad\quad 1.57E-5  &~~~\quad\quad 0.96\\ \hline
    ~8193 &51 &~417843 &7.46E-7  &~~1.24&&~~8192 &\quad\quad 7.25E-6  &~~~\quad\quad 1.15\\ \hline
    16385 &55  &~901175 &4.92E-7  &~~0.54&&~16384 &\quad\quad 2.72E-6  &~~~\quad\quad 1.41\\ \hline
    32769 &61  &1998909 &1.87E-7 &~~1.21&&~32768 &\quad\quad 2.01E-6  &~~~\quad\quad 0.44\\ \hline
    65537 &65  &4259905 &9.34E-8 &~~0.92&&~65536 &\quad\quad 9.06E-7  &~~~\quad\quad 1.15\\
    	\hline
	\end{tabular}
\end{table}

\begin{table}[h]
	\small{\caption{\emph{Example $f_4$~--Numerical results}}}
	\label{Tab4}	\centering
	\begin{tabular}{llllll|lll}
		\hline
		\multicolumn{5}{c}{Our method }& &\multicolumn{3}{c}{The Owen-scrambled Sobol' method} \\
		\hline
		 ~~~$M$   &~$t$  &~~~~$Mt$  &~~MSE &Order&~  &~~~$M$ &~~\quad\quad\quad MSE &~~\quad\quad Order  \\
		\hline
		~~~~5 &~3  &~~~~~15 &~1.99E-1 & && ~~~~~4 &\quad\quad1.41E-1  &~~~\quad\quad\\ \hline
        ~~~~9 &~7  &~~~~~63 &~1.60E-2 &~~1.85 && ~~~~~8 &\quad\quad6.87E-2  &~~~\quad\quad1.04\\ \hline
        ~~~17 &~9  &~~~~153 &~2.01E-3 &~~2.48&& ~~~~16 &\quad\quad3.18E-2  &~~~\quad\quad1.11\\ \hline
        ~~~33 &13  &~~~~429 &~5.81E-5 &~~3.53&& ~~~~32 &\quad\quad1.41E-2  &~~~\quad\quad1.18\\ \hline
        ~~~65 &17  &~~~1105 &~3.55E-6 &~~3.00&& ~~~~64 &\quad\quad7.94E-3  &~~~\quad\quad0.82\\ \hline
	~~129 &21  &~~~2709 &~1.92E-7&~~3.28 &&~~~128 &\quad\quad4.54E-3  &~~~\quad\quad0.81\\ \hline
	~~257 &25  &~~~6425 &~1.42E-8&~~3.03&&~~~256 &\quad\quad1.78E-3  &~~~\quad\quad1.35\\ \hline
	~~513 &31 &~~16416 &~1.13E-9 &~~2.80&&~~~512 &\quad\quad7.97E-4  &\quad\quad\quad~1.16\\ \hline
	~1025 &35  &~~35875 &5.98E-11 &~~3.62&&~~1024 &\quad\quad5.66E-4  &~~~\quad\quad0.49\\ \hline
    ~2049 &41  &~~84009 &1.25E-11 &~~1.84 &&~~2048 &\quad\quad2.27E-4  &~~~\quad\quad1.32\\ \hline
	~4097 &45  &~184365 &1.02E-12 &~~3.19&&~~4096 &\quad\quad1.14E-4  &~~~\quad\quad0.99\\ \hline
    ~8193 &51 &~417843 &8.60E-14  &~~3.02&&~~8192 &\quad\quad6.29E-5  &~~~\quad\quad0.86\\ \hline
    16385 &55  &~901175 &1.13E-14  &~~2.64&&~16384 &\quad\quad2.29E-5  &~~~\quad\quad1.46\\ \hline
    32769 &61  &1998909 &6.00E-16 &~~3.69&&~32768 &\quad\quad4.53E-6  &~~~\quad\quad2.34\\ \hline
    65537 &65  &4259905 &8.31E-17 &~~2.61&&~65536 &\quad\quad5.89E-7  &~~~\quad\quad2.94\\
    	\hline
	\end{tabular}
\end{table}
\begin{figure}
\centering
\includegraphics[width=0.6\textwidth]{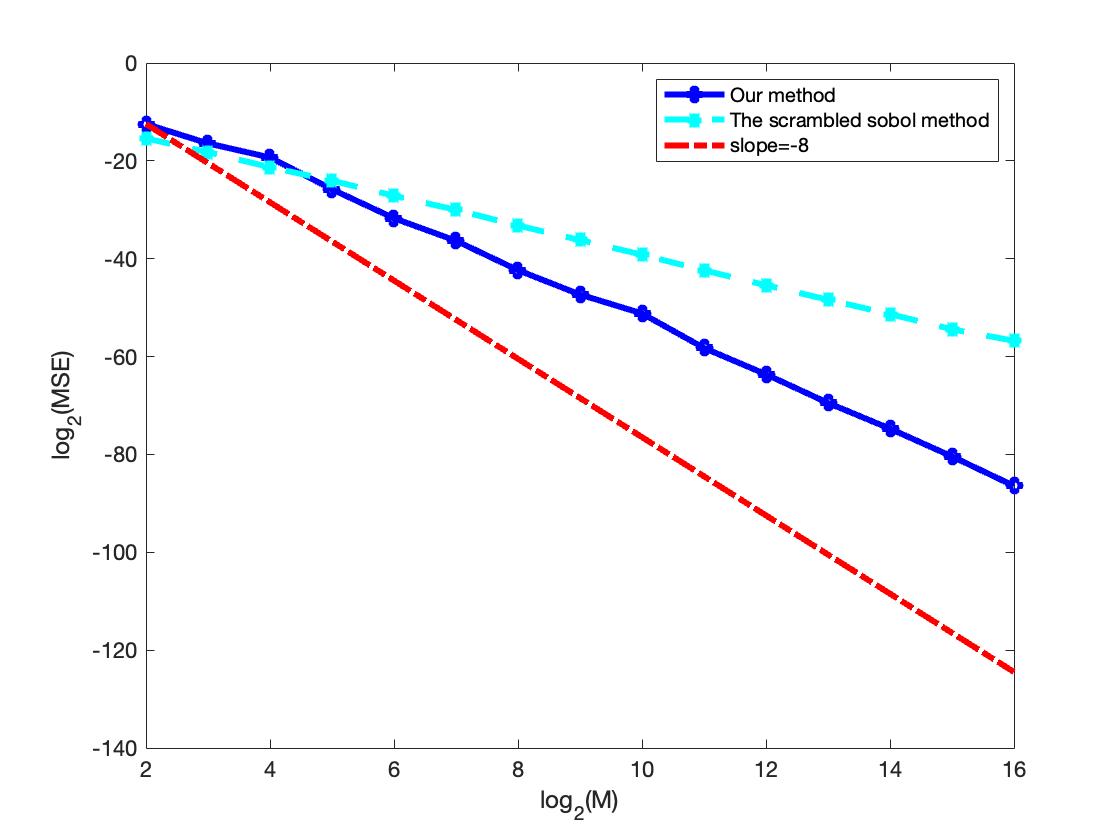} 
\caption{Convergence behavior of our method for $f_{1}$ with $M=2L+1$, $L= 2^k\,(k=1,2,\dots,15)$, and $r=L/\sqrt{8\log(2L+1)}$. The Owen-scrambled Sobol' method employs sample sizes of $M=2^k$ for $k=2,3,\dots,16$. The  MSE are computed based on $100$ runs. The theoretically optimal convergence order is $1/M^{8-\epsilon}$, where $\epsilon$ is an arbitrarily small positive number.} \label{fi1}
\end{figure}
\begin{figure}
\centering
\includegraphics[width=0.6\textwidth]{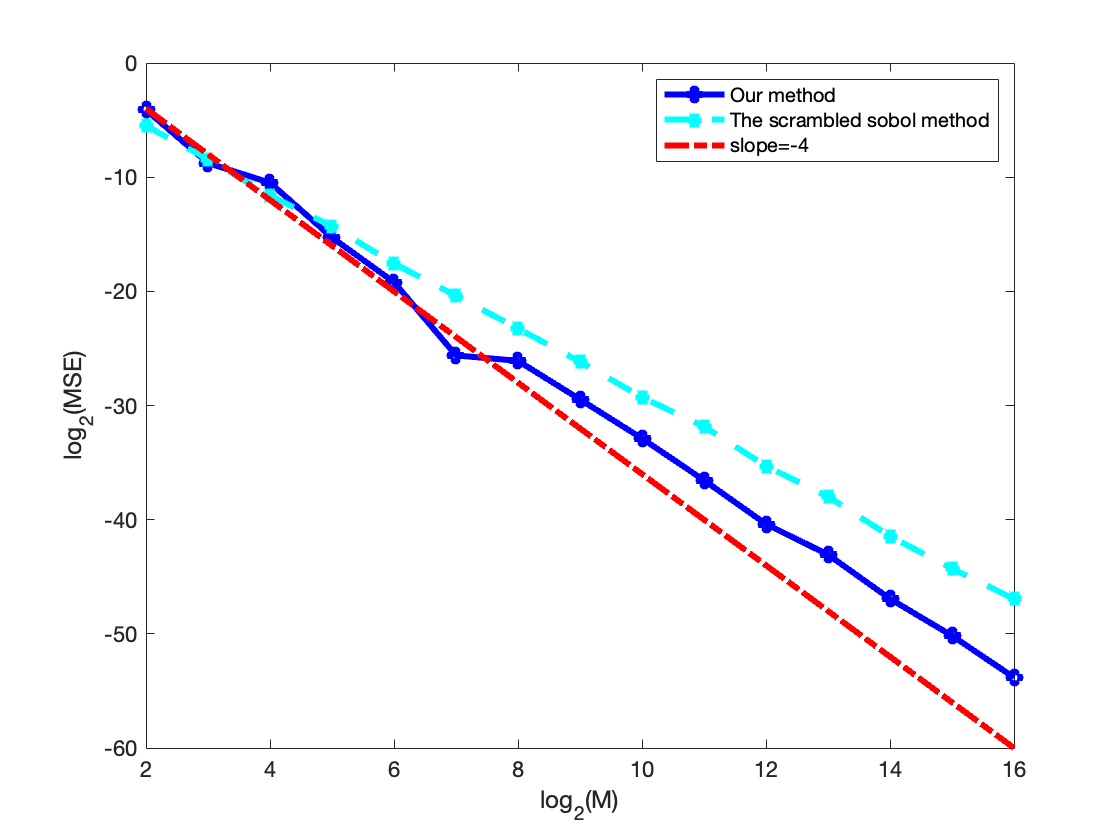} 
\caption{Convergence behavior of our method for $f_{2}$ with sample size $M=2L+1$, $L= 2^k\,(k=1,2,\dots,15)$, and $r=L/\sqrt{4\log(2L+1)}$. The Owen-scrambled Sobol' method employs sample sizes of $M=2^k$ for $k=2,3,\dots,16$. The MSE are computed based on $100$ runs. The theoretically optimal convergence order is $1/M^{4-\epsilon}$, where $\epsilon$ is an arbitrarily small positive number.}  \label{fi2}
\end{figure}

\begin{figure}
\centering
\includegraphics[width=0.6\textwidth]{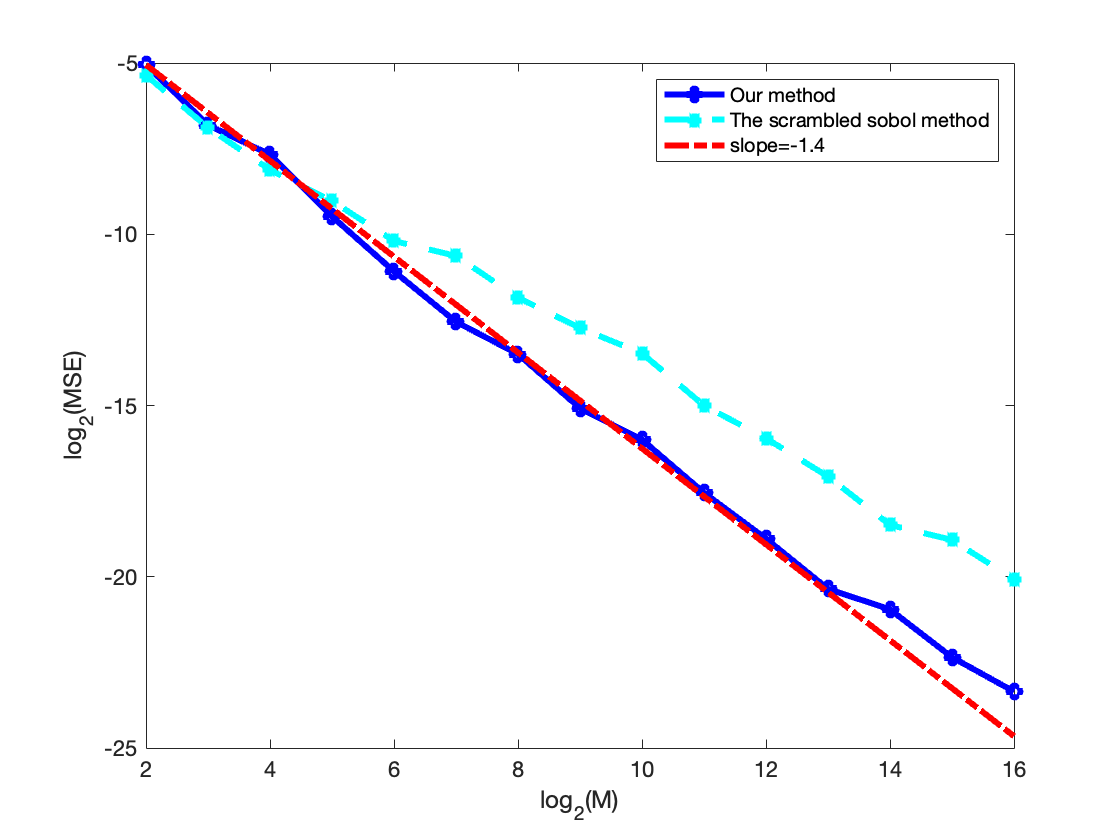} 
\caption{Convergence behavior of our method for $f_{3}$ with sample size $M=2L+1$, $L= 2^k\,(k=1,2,\dots,15)$, and $r=L/\sqrt{2\log(2L+1)}$. The Owen-scrambled Sobol' method employs sample sizes of $M=2^k$ for $k=2,3,\dots,16$. The MSE are computed based on $100$ runs.}  \label{fi3}
\end{figure}

\begin{figure}
\centering
\includegraphics[width=0.6\textwidth]{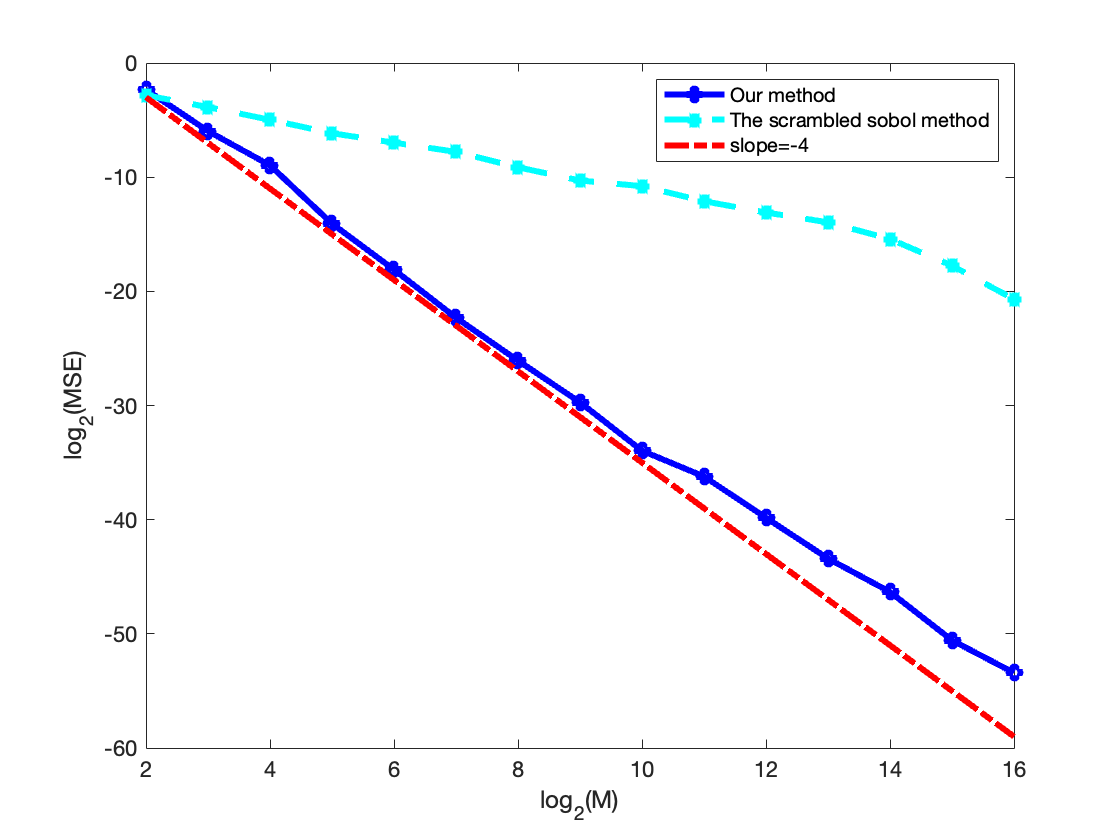} 
\caption{Convergence behavior of our method for $f_{4}$ with sample size $M=2L+1$, $L= 2^k\,(k=1,2,\dots,15)$, and $r=L/\sqrt{4\log(2L+1)}$. The Owen-scrambled Sobol' method employs sample sizes of $M=2^k$ for $k=2,3,\dots,16$. The MSE are computed based on $100$ runs.}  \label{fi4}
\label{fig:example}  
\end{figure}
Fig \ref{fi1} presents the numerical results of $f_1$, obtained with different values of $M$, revealing that the convergence order of our method exceeds $1/M^5$. The numerical results of $f_2$ shown in Fig \ref{fi2}, demonstrate a convergence order approaching $1/M^3$. For  test functions $f_1$ and $f_2$, our convergence rate is comparable to that reported in \cite{dick2022component} and \cite{goda2024simple}, achieved without any prior knowledge of the weights.  The gap between numerical and the theoretical optimal error arises from two  factors: (i) the
influence of the logarithmic factor in the error term, which is non-negligible in non-asymptotic
regimes. In particular, the power of the logarithmic factor increases with $s$, hence, for instance, the actual error for $f_{1}$ exhibits a larger gap from the theoretically optimal error bound than that for $f_{2}$; (ii) the influence of the constant \( C_{\alpha,s,\gamma,N} \) as specified in Theorem \ref{T3}. 
As the weight increases, this constant diminishes, thereby facilitating the numerical error to close to the theoretically optimal error. Consistent with this, numerical examples in \cite{goda2024simple} corroborate that larger weights enhance the alignment between the two.

The Owen-scrambled Sobol' sequence was adopted as the benchmark for comparison with our proposed method. For functions \(f_1\), \(f_2\), and \(f_3\), both methods exhibit comparable performance. Fig. \ref{fi4}  illustrates that the Owen-scrambled Sobol' method is significantly influenced by the high-frequency term introduced in test case \(f_4\) (when augmented with the median trick \cite{pan2023super}, the Owen-scrambled Sobol' method can achieve a super-polynomial convergence rate for integrals of analytic functions. It is important to emphasize that this convergence rate is asymptotic in nature). In contrast, ‌our algorithm is unaffected by high-frequency  term \(\sin(20000\pi x_1)\), provided that the numerical error incurred in evaluating \( \sin(20000\pi x_1) \) remains sufficiently small.

In Tabs. 1-4, we present the total sample size (including the number of repetitions $t$), and the corresponding errors that correspond to the numerical performance shown in Figs. 1–4, respectively. 
We compute the ``Order''   using the formula specified below: 
\begin{itemize}
\item
For our method, the ``Order'' in $j$-th row is given by:\\
\[
\text{Order}=\frac{\log_{2}\big(\text{MSE}((Mt)_{j-1})/\text{MSE}((Mt)_{j})\big)}{\log_{2}((Mt)_j/(Mt)_{j-1})},
\]
where \((Mt)_j\) and \((Mt)_{j-1}\) denote two values of 
\(Mt\) in $j$-th row and $j-1$-th row, respectively.  
\item For the Owen-scrambled Sobol' method, the   ``Order''  in $j$-th row is calculated as:
\[
\text{Order}=\frac{\log_{2}\big(\text{MSE}(M_{j-1})/\text{MSE}(M_{j})\big)}{\log_{2}(M_j/M_{j-1})},
\]
here, \(M_j\) and \(M_{j-1}\) represent two values of \(M\) in $j$-th row and $j-1$-th row, respectively.
\end{itemize}
The ``Order'' here characterizes the magnitude of the local slope of the error curve, similar to the curves presented in Figs. 1-4. The key difference lies in the abscissa: the abscissa is $\log_{2}(Mt)$ for our method. Due to this difference, the order of our method reported in Tabs. 1–4 exhibits a slight decrease compared to that in Figs. 1–4. Specifically, in Tabs. 1-4, the average orders of the last eight numerical results of our method are $4.72$, $ 2.98$, $ 1.05$, and $2.93$, respectively. As $L$ increases, the impact of the number of repetitions
defined as $t=2\lceil\log_{2}(2L) \log_{2}(\log_{2}(2L))/2\rceil+1$ 
on the convergence order becomes progressively negligible. 

\section{Future Work}

We have theoretically formulated a nearly optimal integration algorithm applicable to periodic isotropic Sobolev spaces and weighted Sobolev spaces, which is characterized by polynomial tractability. However, the upper bounds for the root mean square error (RMSE) in these two spaces involve the factors $\log^{1+\frac{2s}{d}}(M)$ and $\log^{1+2s}(M)$, respectively, where $M$ represents the sample size. This dependency poses challenges to the practical efficacy of the algorithm. We propose that an integration algorithm, which integrates the median trick with a random quasi-Monte Carlo rule, can potentially attain nearly optimal bounds, contingent upon the satisfaction of certain mild conditions. The algorithm is defined as follows:
$$
I(f):=\frac{1}{M}\sum_{0\le l< M} f_{\boldsymbol{\eta},M}\left(\left\{\frac{z-lH_{M}}{M}\right\}\right),
$$
where $M$ is a prime number, and $H_{M}$ and $z$ are sampled uniformly from the sets $\mathbb Z^{d}\cap [1,M)^{d}$ and $\mathbb Z^{d}\cap [0,M)^{d}$, respectively. The function $f_{\boldsymbol{\eta},M}$ is defined in (Eq.~(\ref{MCC})). This algorithm has the potential to reduce the logarithmic factors present in the upper bound. Under the condition \(s \le \Theta(d)\), it is posited that our proposed method may achieve nearly optimality and reduce the logarithmic factors present in the upper bound. Notably, our algorithm diverges from the one presented in \cite{goda2024simple} in two significant ways: (1) it does not utilize random prime selection; and (2) it incorporates non-uniform random shifts. This approach may allow our method to effectively address discontinuous functions (specifically when \(s \le 1/2\) as stated in Theorem 5.1) without necessitating prior knowledge of the weights or the smoothness order \(s\). In contrast, the algorithm described in \cite{goda2024simple} is unable to accommodate discontinuous functions, and the findings by \cite{nuyens2023randomised} rely on prior knowledge of the weights. We intend to investigate this algorithm further in future research. 

\section*{Acknowledgments}
The authors express their gratitude to the reviewers for their valuable feedback. Their constructive observations have substantially contributed to the overall quality and comprehensibility of the manuscript.

\end{document}